\documentclass[journal]{IEEEtran}

\ifCLASSINFOpdf
\usepackage[pdftex]{graphicx}
\graphicspath{{../pdf/}{../jpeg/}}
\DeclareGraphicsExtensions{.pdf,.jpeg,.png}
\else

\usepackage{graphicx}
\usepackage[cmex10]{amsmath}
\graphicspath{{../eps/}}
\allowdisplaybreaks

\fi

\usepackage{color}
\usepackage{amssymb}
\usepackage{amsmath}
\usepackage{amsthm}
\usepackage{subfig}
\newtheorem{lem}{Lemma}
\newtheorem{thm}{Theorem}
\begin{document}

\title{Output-Feedback Boundary Control of a Heat PDE Sandwiched Between Two ODEs}

\author{Ji~Wang
        and Miroslav~Krstic,~\IEEEmembership{Fellow,~IEEE}
\thanks{
J. Wang and M. Krstic are with the Department of Mechanical and Aerospace Engineering, University of California, San Diego, La Jolla, CA 92093-0411 USA (e-mail: jiwang9024@gmail.com, krstic@ucsd.edu).}}

\maketitle

\begin{abstract}
We present designs for exponential stabilization of an ODE-heat PDE-ODE
coupled system where the control actuation only acts in one ODE. The combination of PDE backstepping and ODE backstepping is employed in a state-feedback control law and in an observer that estimates PDE and two ODE states only using one PDE boundary measurement. Based on the state-feedback control law and the observer, the output-feedback control law is then proposed. The exponential stability of the closed-loop system and the boundedness and exponential convergence of the control law are proved via Lyapunov analysis. Finally, numerical simulations validate the
effectiveness of this method for the ``sandwiched'' system.
\end{abstract}


\IEEEpeerreviewmaketitle

\section{Introduction}
\paragraph{Control of parabolic PDEs}
Parabolic partial differential equations (PDEs) are predominately used in describing fluid, thermal, and chemical dynamics, including many applications of sea ice melting and freezing \cite{Wettlaufer1991Heat}, continuous
casting of steel \cite{Petrus2012Enthalpy} and lithium-ion batteries \cite{koga2017State}.
These therefore give rise to related important control and estimation problems, i.e., the boundary control and state observation of parabolic PDEs in
\cite{krstic2008Adaptiveboundary,Meurer2009Tracking,Cheng2009sampled,Pisano2012Boundary,2015Boundary,Deutscher2015backstepping,Deutscher2016backstepping,Orlov2017Output} and \cite{Ahmed2016Adaptive,Smyshlyaev2005Backstepping,Smyshlyaev2009Adaptive} respectively.
\paragraph{Control of parabolic PDE-ODE systems}
In addition to the aforementioned works on parabolic PDEs, topics concerning parabolic PDE-ODE coupled systems are
also popular, which have rich physical background such as coupled
electromagnetic, coupled mechanical, and coupled chemical
reactions \cite{tang2011state}. Using the backstepping method, state-feedback and output-feedback control designs of a class of heat PDE-ODE couled systems were presented in \cite{susto2010control,tang2011stabilization,tang2011state}. The problem of state-observation is addressed for some parabolic PDE-ODE models in \cite{Ahmed2015Observer,Tang2017State}. The sliding model control was proposed to achieve boundary feedback stabilization of a heat PDE-ODE
cascade system with external disturbances in \cite {wang2015Sliding}.
\paragraph{Control of ODE-PDE-ODE systems}
All aforementioned works consider actuation of PDE boundaries and ignore the dynamics of the actuator. However, sometimes the actuator dynamics may not be neglected, especially when dominant time constants of the actuator are closed to those of the plant. Considering the parabolic PDE-ODE coupled system with ODE actuator dynamics, it gives rise to a more challenging control/estimation problem of an ODE-PDE-ODE ``sandwiched'' system. Fewer attempts have been made on the boundary
control of such an ODE-PDE-ODE system or PDE systems following ODE actuator dynamics where the controller acts. The boundary control of a viscous
Burgers' equation with an integration at the input, which is
regarded as a first-order linear ODE in the input channel,
was considered in \cite{Liu2000Backstepping}. Backstepping state-feedback control design for a transport PDE-ODE system where an integration at the input of the transport PDE was proposed in \cite{2008Lyapunov}. The control problem of an ODE with input delay and unmodeled bandwidth-limiting actuator dynamics, which is represented by an ODE-transport PDE-ODE system where the input ODE is first order, is successfully addressed in \cite{krstic2010compensating}. Stabilization of $2\times2$ coupled linear first-order hyperbolic PDEs sandwiched around two ODEs was also achieved in \cite{wangcontrol}.

In this paper, we use the combination of ODE backstepping and PDE backstepping methods to exponentially stabilize an ODE-heat PDE-ODE coupled system where the two ODEs are of arbitrary orders. An observer is designed to estimate all PDE and ODE states only using one PDE boundary value and then the observer-based output feedback control law is proposed.
\paragraph{Main contributions}
\begin{itemize}
\item This is the first result of stabilizing such an ODE-parabolic PDE-ODE ``sandwiched'' system where the control action only acts in one ODE.
\item Compared with our previous work \cite{wangcontrol} where a state-feedback control law was designed to stabilize the ODE-hyperbolic PDE-ODE system with the second-order input ODE and only a sketch of the design and analysis for an arbitrary-order input ODE was provided, in the present paper we extend the hyperbolic PDE to a parabolic PDE, where challenges appear because of the higher order spatial derivative and the inconformity between the orders of time and spatial derivatives. In addition, we design an observer to estimate all the states of the ODE-PDE-ODE system only using one PDE boundary value, and an output-feedback control law is proposed. Moreover, more detailed control design and stability analysis of the system where arbitrary-order
ODEs sandwich around the PDE are presented.
\item Compared with the previous results about stabilizing ODE-transport PDE-ODE systems where both ODEs are first order \cite{2008Lyapunov,krstic2010compensating,AnfinsenStabilization}, in addition to replacing the transport PDE by a heat PDE, we achieve a more challenging and general result where the orders of both ODEs sandwiching the PDE are arbitrary.
\end{itemize}
\paragraph{Organization}
The rest of the paper is organized as follows. The concerned model is presented in Section \ref{sec:Problem}. The state-feedback control design combining the PDE backstepping and ODE backstepping is shown in Section \ref{eq:sp}. The observer design and the output-feedback control law with stability analysis of the closed-loop system are proposed in Section \ref{sec:of}. The simulation results are provided in Section \ref{sec:sim}. The conclusion and future work are presented in Section \ref{sec:con}.

Throughout this paper, the partial derivatives and total derivatives are denoted as:
\begin{align*}
{f_x}(x,t) =& \frac{{\partial f}}{{\partial x}},~~ \quad {f_t}(x,t) = \frac{{\partial f}}{{\partial t}},\notag\\
\partial_x^m f(x,t)&=\frac{{\partial^m f}}{{\partial x^m}},~~ \partial_t^m f(x,t)=\frac{{\partial^m f}}{{\partial t^m}}, \notag\\
f '(x) =& \frac{{d f}}{{dx}},\quad \dot f(t) = \frac{{df}}{{dt}}, \notag\\
d_x^m f(x)&=\frac{{d^m f}}{{d x^m}},~~ d_t^m f(t)=\frac{{d^m f}}{{d t^m}}.
\end{align*}
\section{Problem statement}\label{sec:Problem}
We consider the following system where two ODEs sandwich around a heat PDE as,
\begin{align}
\dot X(t) &= AX(t) + B{u_x}(0,t),\label{eq:o1}\\
{u_t}(x,t) &= q{u_{xx}}(x,t),\label{eq:o2}\\
u(0,t) &= {C_X}X(t),\label{eq:o3}\\
u(1,t) &= {C_z}Z(t),\label{eq:o4}\\
\dot Z(t) &= {A_z}Z(t) + {B_z}U(t),\label{eq:o5}
\end{align}
$\forall (x,t) \in [0,1]\times[0,\infty)$,
where $X(t) \in \mathbb{R}^{n\times1}$, $Z(t)\in \mathbb{R}^{m\times1}$ are ODE states, $n,m\in N^*$, $N^*$ denoting positive integers. $u(x,t)\in R$ are states of the PDE.
$\emph{A}\in \mathbb{R}^{n\times n}$, $\emph{B}\in \mathbb{R}^{n\times1}$ satisfy that the pair $[A,B]$ is controllable. ${C_X}\in \mathbb{R}^{1\times n}$ and $q\in R$ are arbitrary. $A_z\in\mathbb{R}^{m\times m}$ is
\begin{align}
{A_z} = \left[ {\begin{array}{*{20}{c}}
0&1&0&0& \cdots &0\\
0&0&1&0& \cdots &0\\
{}&{}& \vdots &{}&{}&{}\\
0&0&0&0& \cdots &1\\
{{\bar a_1}}&{{\bar a_2}}&{{\bar a_3}}& \cdots &{{\bar a_{m - 1}}}&{{\bar a_m}}
\end{array}} \right]_{m\times m},\label{eq:Az}
\end{align}
where $\bar a_{1},\cdots,\bar a_{m}$ are arbitrary constants.
$B_z=[0,0,\cdots,1]^T\in \mathbb{R}^{m\times1}$, $C_z=[1,0,\cdots,0]\in \mathbb{R}^{1\times m}$. {Note that $(A_z,B_z)$ and $(A_z,C_z)$ are in the controllability normal form and observability normal form respectively.} $U(t)$ is the control input to be designed.

The control objective is to exponentially stabilize all ODE states $Z(t),X(t)$ and PDE states $u(x,t)$ by designing a control input $U(t)$ in one ODE, and  control input itself should be guaranteed exponentially convergent as well.

The control design in this paper can be applied in the Stefan problem describing the melting or solidification mechanism with liquid-solid dynamics \cite{koga2016Backstepping}, i.e., heat PDE-ODE , driven by a thermal actuator described by an ODE at the boundary of the liquid phase.
\section{State-Feedback Design}\label{eq:sp}
In this section, we combine the PDE backstepping (Section \ref{sec:TandB}) and the ODE backstepping (Section \ref{control}) to design a state-feedback control law. Exponential stability of the state-feedback closed-loop system is proved in Section \ref{sec:stable}. The boundedness and exponential convergence of the state-feedback control law is proved in Section \ref{sec:EBU(t)}.
\subsection{Backstepping for PDE-ODE}\label{sec:TandB}
We would like to use the infinite-dimensional backstepping transformations \cite{krstic2009compensating}:
\begin{align}
w (x,t) = u(x,t) - \int_0^x {{\phi}}(x,y)u(y,t)dy- {\Phi}(x)X(t),\label{eq:t1}
\end{align}
with the inverse transformation as
\begin{flalign}
u (x,t) = w(x,t) - \int_0^x {{\psi}}(x,y)w(y,t)dy- {\Gamma}(x)X(t),\label{eq:tI}
\end{flalign}
to convert the original system \eqref{eq:o1}-\eqref{eq:o3} to
\begin{align}
\dot X(t) &= (A+BK)X(t) + B{w_x}(0,t),\label{eq:target1}\\
{w_t}(x,t) &= q{w_{xx}}(x,t),\label{eq:target2}\\
w(0,t) &= 0,\label{eq:target3}
\end{align}
where the right boundary condition $w(1,t)$ will be given later. $A + BK$ is Hurwitz by choosing the control parameter $K$ since $(A,B)$ is assumed controllable.

{Mapping the original system \eqref{eq:o1}-\eqref{eq:o3} and the system \eqref{eq:target1}-\eqref{eq:target3} via the transformations \eqref{eq:t1}-\eqref{eq:tI}, the explicit solutions of kernels in \eqref{eq:t1}-\eqref{eq:tI} can be obtained as
\begin{align}
{\phi}(x,y)&=\left[C_X,K-C_XBC_X\right]e^{D(x-y)}\left[\begin{array}{c}
                                        I \\
                                        0
                                      \end{array}
\right]B,\\
\Phi(x)&=\left[C_X,K-C_XBC_X\right]e^{D(x-y)}\left[\begin{array}{c}
                                        I \\
                                        0
                                      \end{array}
\right],\\
\psi(x,y)&=\left[C_X,K\right]e^{E(x-y)}\left[\begin{array}{c}
                                        I \\
                                        0
                                      \end{array}
\right]B,\\
\Gamma(x)&=\left[C_X,K\right]e^{E(x-y)}\left[\begin{array}{c}
                                        I \\
                                        0
                                      \end{array}
\right],
\end{align}
$0\le y\le x\le 1$, where the matrix $I$ denotes the identity matrix of the appropriate dimension and $D,E$ are
\begin{align}
D=\left(
    \begin{array}{cc}
      0 & A \\
      I & -BC_X \\
    \end{array}
  \right),~~E=\left(
              \begin{array}{cc}
                0 & A+BK \\
                I & 0 \\
              \end{array}
            \right),
\end{align}
which ensures the invertibility and boundedness of the backstepping transformation \eqref{eq:t1}-\eqref{eq:tI}. The detailed calculations of the kernels are shown in \cite{tang2011state}.
Note that dealing with the right boundary condition in the following steps will not affect determination of the kernels in \eqref{eq:t1}-\eqref{eq:tI}.}

{Let us now calculate the right boundary condition $w(1,t)$ of \eqref{eq:target1}-\eqref{eq:target3}. The right boundary condition $w(1,t)$ can be obtained by taking the $m$-order time derivative of the transformation \eqref{eq:t1} at $x=1$, inserting the original right boundary condition and the inverse transformation \eqref{eq:tI}.

Considering \eqref{eq:o4}-\eqref{eq:o5} with \eqref{eq:Az}, the right boundary condition of the original system can be written as
\begin{align}
\partial _t^mu(1,t)&=\bar a_{1}u(1,t)+\sum_{k=1}^{m-1}\bar a_{k+1}\partial _t^ku(1,t)+U(t)\notag\\
&=\bar a_{1}u(1,t)+\sum_{k=1}^{m-1}\bar a_{k+1}q^k\partial _x^{2k}u(1,t)+U(t).\label{eq:vtt}
\end{align}
Taking the $m$ times derivative in $t$ of \eqref{eq:t1} at $x=1$, we have}
\begin{align}
&\partial _t^mw(1,t) =
\partial _t^mu(1,t) - {q^m}\int_0^1 {\partial _y^{2m}\phi (1,y)u(y,t)} dy\notag\\
&  + q^m\sum\limits_{i = 1}^{2m} {} {( - 1)^i}\partial _y^{i - 1}\phi (1,1)\partial _x^{2m - i}u(1,t)\notag\\
& - \Phi (1){A^m}X(t)-\bigg(q^m\sum\limits_{i = 1}^{2m} {} {( - 1)^i}\partial _y^{i - 1}\phi (1,0)\partial _x^{2m - i}\notag\\
 &+ \Phi (1)\sum\limits_{i = 1}^m {} {A^{i - 1}}B{q^{m - i}}\partial _x^{2(m - i) + 1}\bigg)u(0,t),\label{eq:wmt}
\end{align}
for $m\in N^*$, where
\begin{align*}
\partial _x^{k}u(1,t)\triangleq \partial _x^{k}u(x,t)|_{x=1},~\partial _x^{k}u(0,t)\triangleq \partial _x^{k}u(x,t)|_{x=0}.
\end{align*}
Insert \eqref{eq:vtt} into \eqref{eq:wmt} to replace $\partial _t^mu(1,t)$. Then rewrite $u$ in \eqref{eq:wmt} as $w$ via the inverse transformation \eqref{eq:tI}, where the $k$-order derivative of the inverse transformation \eqref{eq:tI} in $x$ would be used:
\begin{flalign}
\partial _x^ku (x,t)& = \partial _x^k\bigg(w(x,t) - \int_0^x {{\psi}}(x,y)w(y,t)dy- {\Gamma}(x)X(t)\bigg)\notag\\
&=\partial _x^kw(x,t) - \partial _x^k\bigg(\int_0^x {{\psi}}(x,y)w(y,t)dy\bigg)\notag\\
&\quad- d _x^k{\Gamma}(x)X(t)\notag\\
&=\partial _x^kw(x,t)-\int_0^x {\partial _x^{k}{\psi}(x,y)w(y,t)dy}\notag\\
&\quad-\sum\limits_{j = 0}^{k - 1} {} {\chi _{k\_j}}(x)\partial _x^{k - j - 1}w(x,t)- d _x^k{\Gamma}(x)X(t)\label{eq:uxk}
\end{flalign}
for $k\in N^*$,  where ${\chi _{k\_j}}(x)$ denoting the sum of $j$-order derivatives of $\psi(x,x)$ with respect to $x$, results from calculating $\partial _x^k(\int_0^x {{\psi}}(x,y)w(y,t)dy)$, as following
\begin{align}
{\chi _{k\_j}}(x)=\sum_{i=0}^{j}\bar\eta_{k\_j\_i}\partial _x^i\partial _y^{j-i}\psi(x,y)|_{(x,y)=(x,x)},\label{eq:chi}
\end{align}
where constant coefficients $\bar\eta_{k\_j\_i}$ can be easily determined by calculating $\partial _x^k(\int_0^x {{\psi}}(x,y)w(y,t)dy)$ under some specific $k$ according to the order of the plant.

Through plugging \eqref{eq:tI}, \eqref{eq:uxk} into \eqref{eq:wmt} where $\partial _t^mu(1,t)$ has been replaced by \eqref{eq:vtt}, the right boundary condition of the system-$w$ is obtained as:
\begin{align}
&\quad\partial _t^mw(1,t)\notag\\
 &= {{\bar a}_1}w(1,t) - {{\bar a}_1}\int_0^1 {{\psi}(1,y)w(y,t)dy - {{\bar a}_1}{\Gamma}(1)X(t)}\notag\\
 &\quad+ \sum\limits_{k = 1}^{m - 1} {{{\bar a}_{k + 1}}} {q^k}\bigg(\partial _x^{2k}w(1,t) - \int_0^1 {\partial _x^{2k}} {\psi}(1,y)w(y,t)dy\notag\\
 &\quad - \sum\limits_{j = 0}^{2k - 1} {{\chi _{2k\_j}}} (1)\partial _x^{2k - j - 1}w(1,t) - d _x^{2k}{\Gamma}(1)X(t)\bigg)\notag\\
 &\quad- {q^m}\int_0^1 \partial _y^{2m}\phi (1,y)\bigg(w(y,t) - \int_0^y {\psi}(y,z)w(z,t)dz\notag\\
 &\quad - {\Gamma}(y)X(t)\bigg)  dy+ {q^m}\sum\limits_{i = 1}^{2m} {( - 1)^i}\partial _y^{i - 1}\phi (1,1)\notag\\
 &\quad\times\bigg(\partial _x^{2m - i}w(1,t) - \int_0^1 {\partial _x^{2m - i}} {\psi}(1,y)w(y,t)dy\notag\\
  &\quad- \sum\limits_{j = 0}^{2m - i - 1} {{\chi _{2m - i\_j}}} (1)\partial _x^{2m - i - j - 1}w(1,t) \notag\\
  &\quad- d _x^{2m - i}{\Gamma}(1)X(t)\bigg)- \Phi (1){A^m}X(t)\notag\\
 &\quad- {q^m}\sum\limits_{i = 1}^{2m}  {( - 1)^i}\partial _y^{i - 1}\phi (1,0)\bigg(\partial _x^{2m - i}w(0,t)\notag\\
  &\quad- \sum\limits_{j = 0}^{2m - i - 1} {{\chi _{2m - i\_j}}} (0)\partial _x^{2m - i - j - 1}w(0,t)\notag\\
   &\quad- d _x^{2m - i}{\Gamma}(0)X(t)\bigg)+ \Phi (1)\sum\limits_{i = 1}^m {} {A^{i - 1}}B{q^{m - i}}\bigg(\partial _x^{2(m - i) + 1}w(0,t)\notag\\
  &\quad- \sum\limits_{j = 0}^{2(m - i) + 1 - 1} {{\chi _{2(m - i) + 1\_j}}} (0)\partial _x^{2(m - i) + 1 - j - 1}w(0,t)\notag\\
   &\quad- d_x^{2(m - i) + 1}{\Gamma}(0)X(t)\bigg)  + U(t)\notag\\
 &= \bigg[{{\bar a}_1} + \sum\limits_{k = 1}^{m - 1} {{{\bar a}_{k + 1}}} {q^k}\partial _x^{2k} - \sum\limits_{k = 1}^{m - 1} {{{\bar a}_{k + 1}}} {q^k}\sum\limits_{j = 0}^{2k - 1} {{\chi _{2k\_j}}} (1)\partial _x^{2k - j - 1}\notag\\
  &\quad+ {q^m}\sum\limits_{i = 1}^{2m} {} {( - 1)^i}\partial _y^{i - 1}\phi (1,1)\partial _x^{2m - i}\notag\\
   &\quad- {q^m}\sum\limits_{i = 1}^{2m} {} {( - 1)^i}\partial _y^{i - 1}\phi (1,1)\sum\limits_{j = 0}^{2m - i - 1} {{\chi _{2m - i\_j}}} (x)\partial _x^{2m - i - j - 1}\bigg]w(1,t)\notag\\
 &\quad+ \bigg[ - {q^m}\sum\limits_{i = 1}^{2m}  {( - 1)^i}\partial _y^{i - 1}\phi (1,0)\partial _x^{2m - i}\notag\\
 &\quad+{q^m}\sum\limits_{i = 1}^{2m} {} {( - 1)^i}\partial _y^{i - 1}\phi (1,0)\sum\limits_{j = 0}^{2m - i - 1} {{\chi _{2m - i\_j}}} (0)\partial _x^{2m - i - j - 1}\notag\\
 &\quad+ \Phi (1)\sum\limits_{i = 1}^m  {A^{i - 1}}B{q^{m - i}}\partial _x^{2(m - i) + 1}- \Phi (1)\sum\limits_{i = 1}^m  {A^{i - 1}}B{q^{m - i}}\notag\\
  &\quad\quad\times\sum\limits_{j = 0}^{2(m - i) + 1 - 1} {{\chi _{2(m - i) + 1\_j}}} (0)\partial _x^{2(m - i) + 1 - j - 1}\bigg]w(0,t)\notag\\
 &\quad- \int_0^1 {\bigg[{q^m}\sum\limits_{i = 1}^{2m} {} {{( - 1)}^i}\partial _y^{i - 1}\phi (1,1)\partial _x^{2m - i}} {\psi}(1,y) + {{\bar a}_1}{\psi}(1,y)\notag\\
  &\quad+ \sum\limits_{k = 1}^{m - 1} {{{\bar a}_{k + 1}}} {q^k}\partial _x^{2k}{\psi}(1,y) + {q^m}\partial _y^{2m}\phi (1,y)\notag\\
   &\quad- {q^m}\int_y^1 {\partial _y^{2m}\phi (1,z){\psi}(z,y)dz} \bigg]w(y,t)dy\notag\\
 &\quad+ \bigg[{q^m}\sum\limits_{i = 1}^{2m} {} {( - 1)^i}\partial _y^{i - 1}\phi (1,0)d _x^{2m - i}{\Gamma}(0) - \Phi (1){A^m}\notag\\
  &\quad- \Phi (1)\sum\limits_{i = 1}^m {} {A^{i - 1}}B{q^{m - i}}d_x^{2(m - i) + 1}{\Gamma}(0)\notag\\
   &\quad- {q^m}\sum\limits_{i = 1}^{2m} {} {( - 1)^i}\partial _y^{i - 1}\phi (1,1)d _x^{2m - i}{\Gamma}(1)\notag\\
 &\quad- \sum\limits_{k = 1}^{m - 1} {{{\bar a}_{k + 1}}} {q^k}d _x^{2k}{\Gamma}(1) - {{\bar a}_1}{\Gamma}(1)\notag\\
  &\quad+ {q^m}\int_0^1 {\partial _y^{2m}\phi (1,y){\Gamma}(y)} dy\bigg]X(t) + U(t),\label{eq:aftt}
\end{align}
where some typical operators are
\begin{align*}
\partial _x^{k}w(1,t)\triangleq \partial _x^{k}w(x,t)|_{x=1},~\partial _x^{k}w(0,t)\triangleq \partial _x^{k}w(x,t)|_{x=0},\notag\\
\partial _y^{k}\phi(1,1)\triangleq\partial _y^{k}\phi(x,y)|_{(x,y)=(1,1)},~d_x^{k}\Gamma(0)\triangleq d_x^{k}\Gamma(x)|_{x=0}.
\end{align*}
\eqref{eq:aftt} is a $m$ order ODE system-$w(1,t)$ with a number of PDE state perturbation terms. For clarity, \eqref{eq:aftt} can be written as
\begin{align}
\partial _t^mw(1,t)&= \mathcal Bw(1,t)+ \mathcal Cw(0,t)\notag\\
&\quad - \int_0^1\mathcal D(y)w(y,t)dy+ \mathcal EX(t)+U(t).\label{eq:aftt1}
\end{align}
Note that \eqref{eq:aftt1} is the right boundary condition of the system \eqref{eq:target1}-\eqref{eq:target3}. $\mathcal B,\mathcal C,\mathcal D,\mathcal E$ in \eqref{eq:aftt1} correspond to the parts including derivative operators in the square brackets before $w(1,t)$, $w(0,t)$, $w(y,t)$, $X(t)$ in \eqref{eq:aftt}. $\mathcal D(y)$ is a function of $y$ and $\mathcal E$ is a constant vector. Note that
\begin{align*}
\mathcal Bw(1,t)\triangleq (\mathcal Bw(x,t))|_{x=1},~~\mathcal Cw(0,t)\triangleq (\mathcal Cw(x,t))|_{x=0}.
\end{align*}
{
\begin{thm}
Considering the system \eqref{eq:o1}-\eqref{eq:Az} and the backstepping transformation \eqref{eq:t1}-\eqref{eq:tI}, \eqref{eq:aftt} holds for $m\in N^*$ which is the order of the ODE \eqref{eq:o5}.
\end{thm}
\begin{proof}
According to the derivation of \eqref{eq:aftt}, i.e., \eqref{eq:wmt}-\eqref{eq:aftt}, we know \eqref{eq:aftt} is obtained by inserting the plant dynamics \eqref{eq:vtt} and \eqref{eq:uxk} into \eqref{eq:wmt} straightforwardly, so \textbf{the correctness of \eqref{eq:aftt} depends on that of \eqref{eq:uxk} and \eqref{eq:wmt}}. Next, we prove the correctness of \eqref{eq:uxk} and \eqref{eq:wmt} by induction. In detail, in order to prove the statement that \eqref{eq:wmt} and \eqref{eq:uxk} hold for arbitrary positive integers $m,k$, we first prove \eqref{eq:wmt} and \eqref{eq:uxk} are true in the initial cases $m=1,2$, $k=1,2$, and then assume they are true for $m={\bar n}$, $k={\bar n}$ and show they hold for $m={\bar n}+1$, $k={\bar n}+1$ respectively where ${\bar n}$ is a positive integer.

\textbf{Proof of correctness of \eqref{eq:wmt} by induction}:

\textbf{a). Checking the initial cases $m=1,2$.}
It is easily to see that \eqref{eq:wmt} is true when $m=1$. Here we show the checking for the case $m=2$.

Taking twice time derivatives of \eqref{eq:t1}, we have
\begin{align}
{w_{tt}}(1,t)=&{u_{tt}}(1,t) - {q^2}\phi (1,1){u_{xxx}}(1,t) + {q^2}\phi (1,0){u_{xxx}}(0,t)\notag\\
 &+ {q^2}{\phi _y}(1,1){u_{xx}}(1,t) - {q^2}{\phi _y}(1,0){u_{xx}}(0,t)\notag\\
 &- {q^2}{\phi _{yy}}(1,1){u_x}(1,t) + {q^2}{\phi _{yy}}(1,0){u_x}(0,t)\notag\\
 &+ {q^2}{\phi _{yyy}}(1,1)u(1,t) - {q^2}{\phi _{yyy}}(1,0)u(0,t)\notag\\
 & - {q^2}\int_0^1 {{\phi _{yyyy}}(1,y)u(y,t)} dy- \Phi (1){A^2}X(t)\notag\\
 & - \Phi (1)AB{u_x}(0,t) - \Phi (1)B{u_{xt}}(0,t).\label{eq:w2t}
\end{align}
Setting $m=2$ in \eqref{eq:wmt}, we obtain
\begin{align}
&\partial _t^2w(1,t) =
\partial _t^2u(1,t) - {q^2}\int_0^1 {\partial _y^{4}\phi (1,y)u(y,t)} dy\notag\\
&  + q^2\sum\limits_{i = 1}^{4} {} {( - 1)^i}\partial _y^{i - 1}\phi (1,1)\partial _x^{4 - i}u(1,t)\notag\\
& - \Phi (1){A^2}X(t)-q^2\sum\limits_{i = 1}^{4} {} {( - 1)^i}\partial _y^{i - 1}\phi (1,0)\partial _x^{4 - i}u(0,t)\notag\\
 &+ \Phi (1)\sum\limits_{i = 1}^2 {} {A^{i - 1}}B{q^{2 - i}}\partial _x^{2(2 - i) }u_x(0,t),\label{eq:wmt2}
\end{align}
where ${q^{2 - i}}\partial _x^{2(2 - i) }u_x(0,t)=\partial _t^{(2-i) }u_x(0,t)$ is used. It can be seen that \eqref{eq:wmt2} is equal to \eqref{eq:w2t}, so \eqref{eq:wmt} is true when $m=2$.

\textbf{{b). Under the induction hypothesis that \eqref{eq:wmt} holds for $m={\bar n}$, show it also holds for $m={\bar n}+1$.}}

Then we can assume \eqref{eq:wmt} is true for $m={\bar n}$ which is a positive integer and next prove \eqref{eq:wmt} also holds for $m={\bar n}+1$.

Taking the time derivative of \eqref{eq:wmt} with setting $m={\bar n}$ we have
\begin{align}
&\quad\partial _t^{{\bar n} + 1}w(1,t)\notag\\
 &= \partial _t^{{\bar n} + 1}u(1,t) - {q^{{\bar n}}}\int_0^1 {\partial _y^{2{\bar n}}\phi (1,y){u_t}(y,t)} dy \notag\\
 &\quad- \Phi (1){A^{{\bar n}}}\dot X(t)+ \sum\limits_{i = 1}^{2{\bar n}} {{q^{{\bar n}}}} {( - 1)^i}\partial _y^{i - 1}\phi (1,1)\partial _{xt}^{2{\bar n} - i}u(1,t)\notag\\
 &\quad- \bigg[\sum\limits_{i = 1}^{2{\bar n}} {{q^{{\bar n}}}} {( - 1)^i}\partial _y^{i - 1}\phi (1,0)\partial _x^{2{\bar n} - i} \notag\\
 &\quad+ \Phi (1)\sum\limits_{i = 1}^{{\bar n}} {} {A^{i - 1}}B{q^{{\bar n} - i}}\partial _x^{2({\bar n} - i) + 1}\bigg]{u_t}(0,t)\notag\\
 &= \partial _t^{{\bar n} + 1}u(1,t) - {q^{{\bar n} + 1}}\int_0^1 {\partial _y^{2{\bar n}}\phi (1,y){u_{xx}}(y,t)} dy\notag\\
  &\quad- \Phi (1){A^{{\bar n} + 1}}X(t) - \Phi (1){A^{{\bar n}}}B{u_x}(0,t)\notag\\
 &\quad+ \sum\limits_{i = 1}^{2{\bar n}} {{q^{{\bar n} + 1}}} {( - 1)^i}\partial _y^{i - 1}\phi (1,1)\partial _x^{2({\bar n} + 1) - i}u(1,t)\notag\\
 &\quad- \bigg[\sum\limits_{i = 1}^{2{\bar n}} {{q^{{\bar n}}}} {( - 1)^i}\partial _y^{i - 1}\phi (1,0)\partial _x^{2{\bar n} - i}\notag\\
  &\quad+ \Phi (1)\sum\limits_{i = 1}^{{\bar n}} {} {A^{i - 1}}B{q^{{\bar n} - i}}\partial _x^{2({\bar n} - i) + 1}\bigg]q{u_{xx}}(0,t)\notag\\
 &= \partial _t^{{\bar n} + 1}u(1,t) - {q^{{\bar n} + 1}}\partial _y^{2{\bar n}}\phi (1,1){u_x}(1,t)\notag\\
 &\quad + {q^{{\bar n} + 1}}\partial _y^{2{\bar n}}\phi (1,0){u_x}(0,t)\notag\\
 &\quad+ {q^{{\bar n} + 1}}\partial _y^{2{\bar n} + 1}\phi (1,1)u(1,t) - {q^{{\bar n} + 1}}\partial _y^{2{\bar n} + 1}\phi (1,0)u(0,t)\notag\\
  &\quad- {q^{{\bar n} + 1}}\int_0^1 {\partial _y^{2({\bar n} + 1)}\phi (1,y)u(y,t)} dy\notag\\
 &\quad- \Phi (1){A^{{\bar n} + 1}}X(t) - \Phi (1){A^{{\bar n}}}B{u_x}(0,t)\notag\\
 &\quad+ \sum\limits_{i = 1}^{2{\bar n}} {{q^{{\bar n} + 1}}} {( - 1)^i}\partial _y^{i - 1}\phi (1,1)\partial _x^{2({\bar n} + 1) - i}u(1,t)\notag\\
 &\quad- \bigg[\sum\limits_{i = 1}^{2{\bar n}} {{q^{{\bar n}}}} {( - 1)^i}\partial _y^{i - 1}\phi (1,0)\partial _x^{2{\bar n} - i}\notag\\
  &\quad+ \Phi (1)\sum\limits_{i = 1}^{{\bar n}} {} {A^{i - 1}}B{q^{{\bar n} - i}}\partial _x^{2({\bar n} - i) + 1}\bigg]q{u_{xx}}(0,t)\notag\\
 &= \partial _t^{{\bar n} + 1}u(1,t) + {q^{{\bar n} + 1}}\partial _y^{2{\bar n}}\phi (1,0){u_x}(0,t)\notag\\
 &\quad- {q^{{\bar n} + 1}}\partial _y^{2{\bar n} + 1}\phi (1,0)u(0,t) \notag\\
 &\quad- {q^{{\bar n} + 1}}\int_0^1 {\partial _y^{2({\bar n} + 1)}\phi (1,y)u(y,t)} dy\notag\\
 &\quad- \Phi (1){A^{{\bar n} + 1}}X(t) - \Phi (1){A^{{\bar n}}}B{u_x}(0,t)\notag\\
 &\quad+ \sum\limits_{i = 1}^{2({\bar n} + 1)} {{q^{{\bar n} + 1}}} {( - 1)^i}\partial _y^{i - 1}\phi (1,1)\partial _x^{2({\bar n} + 1) - i}u(1,t)\notag\\
 &\quad- \bigg[\sum\limits_{i = 1}^{2{\bar n}} {{q^{{\bar n} + 1}}} {( - 1)^i}\partial _y^{i - 1}\phi (1,0)\partial _x^{2{\bar n} - i + 2} \notag\\
 &\quad+ \Phi (1)\sum\limits_{i = 1}^{{\bar n}} {} {A^{i - 1}}B{q^{{\bar n} - i + 1}}\partial _x^{2({\bar n} - i) + 2 + 1}\bigg]u(0,t),\label{eq:www}
 \end{align}
 where $- {q^{\bar n + 1}}\partial _y^{2\bar n}\phi (1,1){u_x}(1,t)$,${q^{\bar n + 1}}\partial _y^{2\bar n + 1}\phi (1,1)u(1,t)$ are combined into  $\sum_{i = 1}^{2(\bar n + 1)} {{q^{\bar n + 1}}} {( - 1)^i}\partial _y^{i - 1}\phi (1,1)\partial _x^{2(\bar n + 1) - i}u(1,t)$ as the terms $i=2\bar n+1$ and $i=2\bar n+2$ respectively.

 Combining ${q^{\bar n + 1}}\partial _y^{2\bar n}\phi (1,0){u_x}(0,t)$,$- {q^{\bar n + 1}}\partial _y^{2\bar n + 1}\phi (1,0)u(0,t)$ into $\sum_{i = 1}^{2(\bar n+1)} {{q^{\bar n + 1}}} {( - 1)^i}\partial _y^{i - 1}\phi (1,0)\partial _x^{2(\bar n+1) - i}u(0,t)$ as the terms $i=2\bar n+1$ and $i=2\bar n+2$ respectively, combining $- \Phi (1){A^{\bar n}}B{u_x}(0,t)$ into $-\Phi (1)\sum_{i = 1}^{\bar n+1} {} {A^{i - 1}}B{q^{\bar n+1 - i }}\partial _x^{2(\bar n+1 - i) + 1}u(0,t)$ as the term $i=\bar n+1$, from \eqref{eq:www} we have
 \begin{align}
&\quad\partial _t^{\bar n + 1}w(1,t)\notag\\
&= \partial _t^{\bar n + 1}u(1,t) - {q^{\bar n + 1}}\int_0^1 {\partial _y^{2(\bar n + 1)}\phi (1,y)u(y,t)} dy\notag\\
 &\quad+ {{q^{\bar n + 1}}}\sum\limits_{i = 1}^{2(\bar n + 1)}  {( - 1)^i}\partial _y^{i - 1}\phi (1,1)\partial _x^{2(\bar n + 1) - i}u(1,t)\notag\\
 &\quad- \Phi (1){A^{\bar n + 1}}X(t)\notag\\
 &\quad- {{q^{\bar n + 1}}}\bigg[\sum\limits_{i = 1}^{2(\bar n+1)}  {( - 1)^i}\partial _y^{i - 1}\phi (1,0)\partial _x^{2(\bar n+1) - i} \notag\\
 &\quad+ \Phi (1)\sum\limits_{i = 1}^{\bar n+1} {} {A^{i - 1}}B{q^{\bar n+1 - i}}\partial _x^{2(\bar n+1 - i)+ 1}\bigg]u(0,t).\label{eq:wm1}
\end{align}
Compare \eqref{eq:wm1} with \eqref{eq:wmt} where $m=\bar n$, we prove \eqref{eq:wmt} is true for $m\in N^*$.

\textbf{Proof of correctness of \eqref{eq:uxk} by induction}:

\textbf{a). Checking the initial cases $k=1,2$}. It is easily to see that \eqref{eq:uxk} is true when $k=1$. Here we show the checking for the case $k=2$.

Taking twice time derivatives of \eqref{eq:tI}, we have
\begin{align}
u_{xx} (x,t) =& w_{xx}(x,t)-\int_0^x \partial _{xx}{\psi}(x,y)w(y,t)dy \notag\\
&-  {{\psi}_x}(x,x)w(x,t)-{{\psi}_y}(x,x)w(x,t)\notag\\
&-{{\psi}}(x,x)w_x(x,t)- {\Gamma}''(x)X(t).\label{eq:tItt}
\end{align}
Setting $k=2$ in \eqref{eq:uxk}
\begin{align}
\partial _x^2u (x,t)& =\partial _x^2w(x,t)-\int_0^x {\partial _x^{2}{\psi}(x,y)w(y,t)dy}\notag\\
&\quad-\sum\limits_{j = 0}^{1}  {\chi _{2\_j}}(x)\partial _x^{1 - j}w(x,t)- d_x^2{\Gamma}(x)X(t),\label{eq:uxk2}
\end{align}
where $\sum\limits_{j = 0}^{1}  {\chi _{2\_j}}(x)\partial _x^{1 - j}w(x,t)={\chi _{2\_0}}(x)w_x(x,t)+{\chi _{2\_1}}(x)w(x,t)$ can be written as $-{{\psi}}(x,x)w_x(x,t)+{{\psi}_x}(x,x)w(x,t)-{{\psi}_y}(x,x)w(x,t)$ with $\bar\eta_{2\_0\_0}=1$, $\bar\eta_{2\_1\_0}=1$ and $\bar\eta_{2\_1\_1}=1$,
according to the definition of ${\chi _{k\_j}}(x)$ \eqref{eq:chi}.

Comparing \eqref{eq:tItt} with \eqref{eq:uxk2}, we know \eqref{eq:uxk} is true when $k=2$.

\textbf{{b). Under the induction hypothesis that \eqref{eq:uxk} holds for $k=\bar n$, show it also holds for $k=\bar n+1$.}}

Taking the derivative of \eqref{eq:uxk} with setting $k=\bar n$ in $x$, we have
\begin{flalign}
&\quad\partial _x^{\bar n+1}u (x,t)\notag\\
&=\partial _x^{\bar n+1}w(x,t)-\int_0^x {\partial _x^{\bar n+1}{\psi}(x,y)w(y,t)dy}\notag\\
&\quad-\partial _x^{\bar n}{\psi}(x,x)w(x,t)-\sum\limits_{j = 0}^{\bar n - 1} {} {\chi_{{\bar n\_j}}}'(x)\partial _x^{\bar n - j - 1}w(x,t)\notag\\
&\quad-\sum\limits_{j = 0}^{\bar n - 1} {} {\chi _{\bar n\_j}}(x)\partial _x^{\bar n+1 - j - 1}w(x,t)- d _x^{\bar n+1}{\Gamma}(x)X(t)\notag\\
&=\partial _x^{\bar n+1}w(x,t)-\int_0^x {\partial _x^{\bar n+1}{\psi}(x,y)w(y,t)dy}\notag\\
&\quad-\sum\limits_{j = 0}^{\bar n} {} {\chi _{\bar n+1\_j}}(x)\partial _x^{\bar n+1 - j - 1}w(x,t)- d_x^{\bar n+1}{\Gamma}(x)X(t),
\label{eq:uxk1}
\end{flalign}
where
\begin{align}
&\quad\sum\limits_{j = 0}^{\bar n} {} {\chi _{{\bar n+1}\_j}}(x)\partial _x^{\bar n+1 - j - 1}w(x,t)\notag\\
&=-\partial _x^{\bar n}{\psi}(x,x)w(x,t)-\sum\limits_{j = 0}^{\bar n - 1} {} {\chi _{\bar n\_j}}'(x)\partial _x^{\bar n - j - 1}w(x,t)\notag\\
&\quad-\sum\limits_{j = 0}^{\bar n - 1} {} {\chi _{\bar n\_j}}(x)\partial _x^{\bar n+1 - j - 1}w(x,t)
\end{align}
with
\begin{align*}
{\chi _{{\bar n+1}\_{\bar n}}}(x)&=-\partial _x^{\bar n}{\psi}(x,x)+{\chi _{\bar n\_{\bar n-1}}}'(x),\notag\\
{\chi _{{\bar n+1}\_0}}(x)&={\chi _{\bar n\_0}}(x),\notag\\
{\chi _{{\bar n+1}\_j}}(x)&={\chi _{\bar n\_{j-1}}}'(x)+{\chi _{\bar n\_j}}(x),~~ j=1,\cdots,\bar n-1.
\end{align*}
Comparing \eqref{eq:uxk1} with \eqref{eq:uxk} where $k=\bar n$, we prove \eqref{eq:uxk} is true for $k\in N^*$.

Because \eqref{eq:wmt}, \eqref{eq:uxk} are true for $m,k\in N^*$ respectively, we can conclude \eqref{eq:aftt} is true because \eqref{eq:aftt} is obtained by inserting \eqref{eq:uxk} into \eqref{eq:wmt}. The proof is completed.
\end{proof}}

\subsection{Backstepping for input ODE with PDE state perturbations}\label{control}
The following backstepping transformation \cite{wangcontrol} for the system-$\left(w(1,t),w_t(1,t),\cdots,\partial_t^{m-1} w(1,t)\right)$ \eqref{eq:aftt} is made:
\begin{align}
y_1(t)&=w(1,t),\label{eq:nct1}\\
y_2(t)&=w_t(1,t)+\tau_1[w(1,t)],\label{eq:nct2}\\
&\cdots\notag\\
y_{m}(t)&=\partial_t^{m-1} w(1,t)\notag\\
&\quad+\tau_{m-1}[w(1,t),\cdots,\partial_t^{m-2} w(1,t)],\label{eq:nctm}
\end{align}
where $\tau_1,\cdots,\tau_{m-1}$ defined in the following steps are the virtual controls in the ODE backstepping method.

\textbf{Step.1} We consider a Lyapunov function candidate as
\begin{align}
V_{y1}=\frac{1}{2}y_1(t)^2.\label{eq:nVb1}
\end{align}
Taking the derivative of \eqref{eq:nVb1}, we obtain $\dot V_{y1}=-c_1y_1(t)^2+y_1(t)y_2(t)$ with the choice of
$\tau_1=c_1y_1(t)$, where $c_1$ is a positive constant to be determined later.

\textbf{Step.2} A Lyapunov function candidate is considered as
\begin{align}
V_{y2}=V_{y1}+\frac{1}{2}y_2(t)^2=\frac{1}{2}y_1(t)^2+\frac{1}{2}y_2(t)^2.\label{eq:mVy2}
\end{align}
Taking the derivative of \eqref{eq:mVy2}, we have
\begin{align*}
\dot V_{y2}=-c_1y_1(t)^2+y_1(t)y_2(t)+y_2(t)(y_3(t)-\tau_2+\dot\tau_1).
\end{align*}
Choosing $\tau_2=\dot\tau_1+y_1(t)+c_2y_2(t)$ , we have
\begin{align}
\dot V_{y2}=-c_1y_1(t)^2-c_2y_2(t)^2+y_2(t)y_3(t).
\end{align}

\textbf{Step.3} $\cdots$ \textbf{Step.m-1}

\textbf{Step.m} Similarly, a Lyapunov function candidate is considered as
\begin{align}
V_{ym}&=V_{y_{m-1}}+\frac{1}{2}y_m(t)^2=\frac{1}{2}y_1(t)^2+\frac{1}{2}y_2(t)^2\notag\\
&\quad+\cdots+\frac{1}{2}y_{m-1}(t)^2+\frac{1}{2}y_m(t)^2.\label{eq:Vym}
\end{align}
Taking the derivative of \eqref{eq:Vym}, we have
\begin{align}
\dot V_{ym}&=-c_1y_1(t)^2-c_2y_2(t)^2-\cdots-c_{m-1}y_{m-1}(t)^2\notag\\
&\quad+y_{m-1}(t)y_{m}(t)+y_{m}(t)\dot y_{m}(t).\label{eq:dVym}
\end{align}
Considering \eqref{eq:nctm} and \eqref{eq:aftt1}, \eqref{eq:dVym} can be rewritten as
\begin{align}
&\dot V_{ym}=-c_1y_1(t)^2-c_2y_2(t)^2-\cdots-c_{m-1}y_{m-1}(t)^2\notag\\
&+y_{m-1}(t)y_{m}(t)+y_{m}(t)\bigg(U(t) + \mathcal Bw(1,t)\notag\\
& + \mathcal Cw(0,t)- \int_0^1\mathcal D(y)w(y,t)dy+ \mathcal EX(t)+\dot\tau_{m-1}\bigg),\label{eq:dVym1}
\end{align}
where
\begin{align}
\tau_{m-1}&=c_1y_1^{m-2}(t)+y_1^{m-3}(t)+c_2y_2^{m-3}(t)+y_2^{m-4}(t)\notag\\
&\quad+\cdots+c_{m-1}y_{m-1}(t),~~\forall m\ge4.\label{eq:taum-1}
\end{align}
Note $y_i^n(t)$ denotes $n$ order derivative of $y_i(t),~\forall i=1,\cdots,m$.

Design the control input as:
\begin{align}
U(t)&= - \mathcal Bw(1,t)- \mathcal Cw(0,t)\notag\\
 &\quad-y_{m-1}(t)-\dot\tau_{m-1}-c_my_m(t).\label{eq:Un}
\end{align}
Recalling \eqref{eq:taum-1} and \eqref{eq:nct1}-\eqref{eq:nctm}, we know
\begin{align}
&y_{m-1}(t)+\dot\tau_{m-1}=\sum\limits_{i=0}^{m-1}\alpha_i(c_1,\cdots,c_{m-1})\partial_t^{i}w(1,t),\label{eq:y-d}\\
&c_my_m(t)=c_m\sum\limits_{i=0}^{m-1}\beta_i(c_1,\cdots,c_{m-1})\partial_t^{i}w(1,t),\label{eq:cy}
\end{align}
where $\alpha_i,\beta_i$ are constants depending on $c_1,\cdots,c_{m-1}$.

The control law \eqref{eq:Un} then can be expressed as
\begin{align}
U(t)=&\mathcal Lw(1,t)-\mathcal Cw(0,t),\label{eq:Uno}
\end{align}
where
\begin{align}
\mathcal L=-\mathcal B-\sum\limits_{i=0}^{m-1}q^i(\alpha_i+c_m\beta_i)\partial_x^{2i}. \label{eq:mL}
\end{align}

Submitting \eqref{eq:Un} into \eqref{eq:dVym1}, we get
\begin{align}
\dot V_{ym}&=-c_1y_1(t)^2-c_2y_2(t)^2-\cdots-c_{m}y_{m}(t)^2\notag\\
&\quad+y_{m}(t)\bigg(- \int_0^1\mathcal D(y)w(y,t)dy+ \mathcal EX(t)\bigg),\label{eq:dVymfinal}
\end{align}
where $c_1,\cdots,c_m$ are positive constants to be determined later.
\subsection{Control law and stability analysis }\label{sec:stable}
Substituting the backstepping transformation \eqref{eq:t1} into \eqref{eq:Uno}, we get the control input expressed by the original states:
\begin{align}
U(t)
&=\mathcal Lu(1,t)-\left(\mathcal L\Phi(1)-\mathcal C\Phi(0)\right)X(t)-\mathcal Cu(0,t)\notag\\
&\quad-\mathcal L\int_0^1 {{\phi}}(1,y)u(y,t)dy\notag\\
&\quad+\bar F\left(u(0,t),\cdots,\partial _x^{2m-2}u(0,t)\right),\label{eq:Uo}
\end{align}
where the function $\bar F$ is obtained from
\begin{align}
\bar F=\left(\mathcal C \int_0^x {{\phi}}(x,y)u(y,t)dy\right)\bigg|_{x=0}\label{eq:bF}
\end{align}
with $\mathcal C$ including differential operators $\sum_{i=0}^{2m-1}\partial _x^{i}$ defined before. The pending control parameters $c_1,\cdots,c_m$ included in $\mathcal L$ will be determined in the following stability analysis.
According to the operators $\mathcal L$, $\mathcal C$, we know the signals used in the control law \eqref{eq:Uo}-\eqref{eq:bF} are $\sum_{i=0}^{2m-1}\partial_x^{i}u(1,t)$, $\sum_{i=0}^{2m-1}\partial_x^{i}u(0,t)$, $X(t)$ and $u(x,t)$. In order to ensure the control law is sufficiently regular, we will require the initial value $u(x,0)$ to be in $H^{2m}(0,1)$ which is defined as $H^{2m}(0,1) = \{u|u\in L^2(0,1),u_x\in L^2(0,1),\cdots,\partial_x^{2m-1}u\in L^2(0,1), \partial_x^{2m}u\in L^2(0,1)\}$ for $m\ge 1$, where $L^2(0,1)$ is the usual Hilbert space.
\begin{thm}\label{main}
{Consider the closed-loop system consisting of the plant \eqref{eq:o1}-\eqref{eq:o5} and the control law \eqref{eq:Uo}-\eqref{eq:bF} with some control parameters $c_1,\cdots,c_m$, and initial values $u(x, 0)\in H^{2m}(0,1)$. There exist constants $\Upsilon_s>0$, $\lambda_s>0$ such that
\begin{align}
\Theta(t)\le \Upsilon_s\Theta(0)e^{-\lambda_s t},\label{eq:T1}
\end{align}
where
\begin{align}
\Theta(t)=\left(\|{{u}(\cdot,t)}\|^2+ \|{{u_x}(\cdot,t)}\|^2+ {\left| {Z(t)} \right|^2}+{\left| {X(t)} \right|^2}\right)^{\frac{1}{2}}.\label{eq:Theo2}
\end{align}
$\left\|\cdot\right\|$ denotes the norm on $L^2(0,1)$, i.e., $\left\|u\right\|=\sqrt{\int_0^{1} u(x,t)^2 dx}$ and  $|\cdot|$ denotes the Euclidean norm.}
\end{thm}
\begin{proof}
We start from studying the stability of the target system. The equivalent stability property between the target system and the original system is ensured due to the invertibility of the PDE backstepping transformation \eqref{eq:t1} and the ODE backstepping transformation \eqref{eq:nct1}-\eqref{eq:nctm}.

First, we study the stability of the PDE-ODE subsystem in the target system via Lyapunov analysis. Second, considering the Lyapunov analysis of the input ODE in Section \ref{control}, Lyapunov analysis of the overall ODE-PDE-ODE system is provided, where the control parameters $c_1, c_2,\cdots,c_m$ in the control law \eqref{eq:Uo} are determined.
\subsubsection{Lyapunov analysis for the PDE-ODE system}
Defining
\begin{align}
\Omega_0(t)&=\left\| {{w}(\cdot,t)}\right\|^2+\left\| {{w_x}(\cdot,t)}\right\|^2+ \left| {X(t)} \right|^2,\label{eq:norm1}
\end{align}
consider now a Lyapunov function
\begin{align}
V_1(t) &= {{X}^T}(t)PX(t) + \frac{a_0}{2}\| {w{{(\cdot,t)}}}\|^2+ \frac{a_1}{2}\|{w_x{{(\cdot,t)}}}\|^2\label{eq:V1}
\end{align}
where the matrix $P = {P^T} > 0$ is the solution to the Lyapunov equation $P(A +BK) + {(A +BK)^T}P =  - Q$, for some $Q = {{Q}^T} > 0$. The positive parameters $a_0,a_1$ are to be chosen later.

From \eqref{eq:norm1}, we have
\begin{align}
\theta_{01}\Omega_0(t)\le V_1(t)\le \theta_{02}\Omega_0(t),
\end{align}
where ${\theta_{01}}=\min \left\{ {\lambda _{\min }}(P),\frac{a_0}{2},\frac{a_1}{2}\right\}>0$, ${\theta_{02}} = \max \left\{ {\lambda _{\max }}(P),\frac{a_0}{2},\frac{a_1}{2}\right\}>0$.

Applying Agmon's inequality, Young's inequality and
Cauchy-Schwarz inequality, taking the derivative of $V_1(t)$ along the
trajectories of \eqref{eq:target1}-\eqref{eq:target3}, we have
\begin{align}
&\dot V_1(t)\le  - \bigg(\frac{{{a_1}q}}{2} - \frac{{4|PB{|^2}}}{{{\lambda _{\min }}(Q)}} - \left(\frac{1}{{4{r_0}}}a_0q+ \frac{1}{{4{r_1}}}{a_1}\right)\bigg){w_x}(0,t)^2 \notag\\
 &- \bigg((a_0 - {a_1})q - \left(\frac{1}{{4{r_0}}}a_0q + \frac{1}{{4{r_1}}}{a_1}\right)\bigg){\left\| {{w_x}} \right\|^2}\notag\\
 &- \bigg(\frac{1}{2}{a_1}q - \left(\frac{1}{{4{r_0}}}a_0q + \frac{1}{{4{r_1}}}{a_1}\right)\bigg){\left\| {{w_{xx}}} \right\|^2}\notag\\
  &- \frac{3}{4}{\lambda _{\min }}(Q)|X{(t)|^2} + {r_0}a_0qw{(1,t)^2} + {r_1}{a_1}{w_t}{(1,t)^2},
\label{eq:dV11}
\end{align}
where $-\|w_{xx}\|^2\le 2\|w_x\|^2-w_x(0,t)^2$ obtained from Agmon's inequality \cite{susto2010control,tang2011state} is used. $r_0,r_1$ are positive constants to be chosen later from Young's inequality.

Choosing parameters $a_0,a_1$ to satisfy
\begin{align}
a_1>\frac{{8|PB{|^2}}}{q{{\lambda _{\min }}(Q)}},~a_0>a_1,\label{eq:a1}
\end{align}
with sufficiently large $r_0,r_1$ in \eqref{eq:dV11}, we arrive at
\begin{align}
&\dot V_1(t)\le  - \bigg((a_0 - {a_1})q - \left(\frac{1}{{4{r_0}}}a_0q + \frac{1}{{4{r_1}}}{a_1}\right)\bigg){\left\| {{w_x}} \right\|^2}\notag\\
  &- \frac{3}{4}{\lambda _{\min }}(Q)|X(t)|^2-\bar\xi_a w_x(0,t)^2 \notag\\
  &+ {r_0}a_0qw{(1,t)^2} + {r_1}{a_1}{w_t}{(1,t)^2},
\end{align}
where $\bar\xi_a>0$.

Using Poincar$\acute e$ inequality, we obtain
\begin{align}
&\quad\dot V_1(t)\notag\\
&\le- \frac{1}{5}\bigg((a_0 - {a_1})q - \left(\frac{1}{{4{r_0}}}a_0q + \frac{1}{{4{r_1}}}{a_1}\right)\bigg){\left\| {{w_x}} \right\|^2}\notag\\
&\quad- \frac{4}{5}\bigg((a_0 - {a_1})q - \left(\frac{1}{{4{r_0}}}a_0q + \frac{1}{{4{r_1}}}{a_1}\right)\bigg){\left\| {{w_x}} \right\|^2}\notag\\
  &\quad- \frac{3}{4}{\lambda _{\min }}(Q)|X{(t)|^2}-\bar\xi_a w_x(0,t)^2\notag\\
  &\quad + {r_0}a_0qw{(1,t)^2} + {r_2}{a_1}{w_t}{(1,t)^2}\notag\\
&\le- \frac{1}{5}\bigg((a_0 - {a_1})q - \left(\frac{1}{{4{r_0}}}a_0q + \frac{1}{{4{r_1}}}{a_1}\right)\bigg)({\left\| {{w_x}} \right\|^2}+{\left\| {{w}} \right\|^2})\notag\\
&\quad- \frac{3}{4}{\lambda _{\min }}(Q)|X{(t)|^2} -\bar\xi_a w_x(0,t)^2\notag\\
&\quad+ {r_0}a_0qw{(1,t)^2} + {r_1}{a_1}{w_t}{(1,t)^2}\notag\\
&\le- \lambda_1 V_1(t)-\bar\xi_a w_x(0,t)^2\notag\\
&\quad+ {r_0}a_0qw{(1,t)^2} + {r_1}{a_1}{w_t}{(1,t)^2},\label{eq:dV1final}
\end{align}
for some positive $\lambda_1$.
\subsubsection{Lyapunov analysis for the overall ODE-PDE-ODE system}
Recalling \eqref{eq:V1},\eqref{eq:Vym}, and define a Lyapunov function
\begin{align}
V(t)=V_1(t)+{R_y}V_{ym}(t).\label{eq:Vall}
\end{align}
where $R_y>0$ is to be determined later.

Defining
\begin{align}
\Omega(t)&=\left\| {w(\cdot,t)}\right\|^2+\left\| {w_x(\cdot,t)}\right\|^2+ \left| {X(t)} \right|^2\notag\\&\quad+y_1(t)^2+\cdots+y_m(t)^2,\label{eq:norm2}
\end{align}
we have
\begin{align}
\theta_{1}\Omega(t)\le V(t)\le \theta_{2}\Omega(t)\label{eq:tVt}
\end{align}
with positive constants ${\theta _{1}} ,{\theta _{2}} $.

Taking the derivative of \eqref{eq:Vall} and using \eqref{eq:dV1final} and \eqref{eq:dVymfinal}, we get
\begin{align}
\dot V&\le  - \lambda_1 V_1(t)-\bar\xi_a w_x(0,t)^2+ {r_0}a_0qw{(1,t)^2} + {r_1}{a_1}{w_t}{(1,t)^2}\notag\\
&\quad-{R_y}c_1y_1(t)^2-{R_y}c_2y_2(t)^2-\cdots-{R_y}c_{m}y_{m}(t)^2\notag\\
&\quad+{R_y}y_{m}(t)\bigg(-\int_0^1 {\mathcal D(x){w}(x,t)dx}+ {\mathcal E} X(t)\bigg).\label{eq:dVall1}
\end{align}
Substituting \eqref{eq:nct1}-\eqref{eq:nct2} into \eqref{eq:dVall1}, applying Young's inequality, Cauchy-Schwarz inequality, we have
\begin{align}
\dot V&\le - \frac{\lambda_1}{2} V_1(t)- \left(\frac{1}{2}\lambda_1\theta_{01}-{R_y}\bar r_3\left|\mathcal E\right|^2\right){\left| {X(t)} \right|^2}\notag\\
&\quad-\left(\frac{1}{2}\lambda_1\theta_{01}-{R_y}\bar r_4\max\limits_{0\le x\le1}\{|{\mathcal D}(x)|\}\right)\int_0^1 {{w}{{(x,t)}^2}} dx\notag\\
&\quad-\left({R_y}c_1-2{r_1}{a_1}c_1^2-{r_0}a_0q\right)y_1(t)^2-\left({R_y}c_2-2{r_1}{a_1}\right)y_2(t)^2\notag\\
&\quad-{R_y}c_3y_3(t)^2-\cdots-{R_y}c_{m-1}y_{m-1}(t)^2\notag\\
&\quad-{R_y}\bigg(c_m-\frac{1}{4\bar r_3}-\frac{1}{4\bar r_4}\bigg)y_m(t)^2-\bar\xi_a w_x(0,t)^2.\label{eq:dVall3}
\end{align}
Positive constants $\bar r_3,\bar r_4$ should satisfy
\begin{align}
&\bar r_3<\frac{\lambda_1\theta_{01}}{2{R_y}\left|\mathcal E\right|^2},~\bar r_4<\frac{\lambda_1\theta_{01}}{2{R_y}\max\limits_{0\le x\le1}\{|{\mathcal D}(x)|\}}.
\end{align}
Choose the control parameter $c_m$ in the control law \eqref{eq:Uo} as
\begin{align}
c_m>\frac{1}{4\bar r_3}+\frac{1}{4\bar r_4},\label{eq:conditionc}
\end{align}
for $m>2$, where ${R_y}$ should be chosen as
\begin{align}
R_y>\max\bigg\{\frac{2{r_1}{a_1}c_1^2+{r_0}a_0q}{c_1},\frac{2r_1a_1}{c_2}\bigg\}.
\end{align}
$c_1,\cdots,c_{m-1}$ can be arbitrary positive constants. If $m=2$, $c_m$ should be chosen as
\begin{align*}
c_m>\frac{2r_1a_1}{R_y}+\frac{1}{4\bar r_3}+\frac{1}{4\bar r_4}
\end{align*}
for $m=2$, with choosing $R_y>\max\bigg\{\frac{2{r_1}{a_1}c_1^2+{r_0}a_0q}{c_1}\bigg\}$.

We thus achieve
\begin{align}
\dot{V} \leq - \lambda V-\bar\xi_a w_x(0,t)^2,\label{eq:Vfinal}
\end{align}
for some positive $\lambda$.

Note that $m\ge 2$ in the above-mentioned proof because $w_t(1,t)^2$ is represented by $y_1(t)^2,y_2(t)^2$ in \eqref{eq:dVall3}. If $m=1$, the following procedure can be adopt to deal with $w_t(1,t)^2$. Substituting \eqref{eq:Uno}-\eqref{eq:mL} into \eqref{eq:aftt1}, we have
\begin{align}
\partial _t^mw(1,t)&= - \int_0^1\mathcal D(y)w(y,t)dy+ \mathcal EX(t)\notag\\
  &\quad-\sum\limits_{i=0}^{m-1}(\alpha_i+c_m\beta_i)\partial_t^{i}w(1,t),\label{eq:wtm}
\end{align}
where $q^i\partial_x^{2i}=\partial_t^{i}$ is used according to \eqref{eq:target2}. Applying Cauchy-Schwarz inequality and $m=1$, we have
\begin{align}
\left|\partial _tw(1,t)\right|^2\le \xi_{a}\|w\|^2+\xi_{a}|X(t)|^2+\xi_{a}\left|w(1,t)\right|^2,\label{eq:wtm0}
\end{align}
for some positive $\xi_{a}$.
Therefore, $r_1a_1w_t(1,t)^2$ can be represented by $r_1a_1\xi_{a}\|w\|^2$+$r_1a_1\xi_{a}|X(t)|^2$+$r_1a_1\xi_{a}w(1,t)^2$, where $r_1a_1\xi_{a}|X(t)|^2$, $r_1a_1\xi_{a}\|w\|^2$ can be ``incorporated'' by $-|X(t)|^2$, $-\|w\|^2$ in $\dot V_1$, and  $(r _0a_0q+r_1a_1\xi_{a})w(1,t)^2=(r _0a_0q+r_1a_1\xi_{a})y_1(t)^2$ (plus another term $r _0a_0qw(1,t)^2$) can be ``incorporated'' by $-c_1y_1(t)^2$ in $\dot V_{ym}$ with large enough ${c_1}$. \eqref{eq:Vfinal} can then be obtained as well in the case of $m=1$.

From \eqref{eq:norm2}-\eqref{eq:tVt} and \eqref{eq:Vfinal}, we can obtain the exponential stability result in the sense of $\Omega(t)^{\frac{1}{2}}\le\Upsilon_{\Omega}\Omega(0)^{\frac{1}{2}}e^{-\lambda_{\Omega}t}$ with some positive $\Upsilon_{\Omega},\lambda_{\Omega}$. Moreover, through further analysis, the exponential stability result in the sense of higher-order norms up to $\|\partial_x^{2m}w(\cdot,t)\|$ also can be obtained, which can be clearly seen in the proof of the next lemma.

Using the invertibility between $(y_1(t),\cdots,y_m(t))$ and $(w(1,t),w_t(1,t),\cdots,\partial_t^{m-1}w(1,t))$ via the backstepping transformation \eqref{eq:nct1}-\eqref{eq:nctm}, and the invertibility between the target system-$(w(x,t),X(t))$ and the original system $(u(x,t),X(t))$ via the backstepping transformation \eqref{eq:t1}, recalling \eqref{eq:o4}-\eqref{eq:Az}, we can conclude that the $(u(x,t),X(t),Z(t))$ system is exponentially stable in the sense of \eqref{eq:T1}-\eqref{eq:Theo2}.

The proof of Theorem \ref{main} is completed.
\end{proof}
\subsection{Boundedness and Exponential convergence of the control input $U(t)$}\label{sec:EBU(t)}
In the last subsections, we have proposed the state-feedback control law and proved all PDE and ODE states are exponentially stable in the origin in the state-feedback closed-loop system. In this subsection, we would like to prove the exponential convergence and boundedness of the control input $U(t)$ \eqref{eq:Uo} in the closed-loop system.

To investigate the boundedness and exponential convergence of the control input \eqref{eq:Uo} where the highest-order-derivative terms are $\partial_x^{2m-1} u(0,t)$, $\partial_x^{2m-1} u(1,t)$, we propose a lemma first, where we estimate the $L_2$ norm of the states up to $2m$-order spatial derivatives $\partial_x^{2m} u(x,t)$.
\begin{lem}\label{lem:2}
Consider the closed-loop system consisting of the plant \eqref{eq:o1}-\eqref{eq:o5} and the control input \eqref{eq:Uo}-\eqref{eq:bF} with some control parameters $c_1,\cdots,c_m$, and initial values $u(x,0)\in H^{2m}(0,1)$.

i) There exist constants $\Upsilon_{2,3}>0$ and $\lambda_{2,3}>0$ such that
\begin{align}
\left(\|\partial_x^{2} u(\cdot,t)\|^2+\|\partial_x^{3} u(\cdot,t)\|^2\right)^{\frac{1}{2}}\le \Upsilon_{2,3}e^{-\lambda_{2,3} t},\label{eq:lemma1norm}
\end{align}
where $\Upsilon_{2,3}$ only dependents on initial values. Note that if $m=1$, the initial value $u(x,0)$ should be in $H^{3}(0,1)$ considering \eqref{eq:lemma1norm}. Using the above result, we can then obtain the following result:

ii) There exist constants $\Upsilon_{2m}>0$ and $\lambda_{2m}>0$ such that
\begin{align}
\sum_{i=2}^{2m}\|\partial_x^{i} u(\cdot,t)\|&\le \Upsilon_{2m}e^{-\lambda_{2m} t},\label{eq:lemma2norm}
\end{align}
where $\Upsilon_{2m}$ only depends on initial values.
\end{lem}
\begin{proof}
The proof is shown in the Appendix, where the exponential stability estimates of the target
system in the sense of higher-order norms up to $\|\partial_x^{2m} w(\cdot,t)\|$ are obtained via Lyapunov analysis, and only \eqref{eq:t1}-\eqref{eq:target3}, \eqref{eq:nct1}-\eqref{eq:nctm}, \eqref{eq:norm2}-\eqref{eq:tVt}, \eqref{eq:Vfinal}-\eqref{eq:wtm0} in the main body are used.
\end{proof}

We then prove the exponential convergence and boundedness of the control input $U(t)$ \eqref{eq:Uo} in the following theorem.
\begin{thm}\label{th:ecU}
In the closed-loop system including the plant \eqref{eq:o1}-\eqref{eq:o5} and the control input $U(t)$ \eqref{eq:Uo}, $|U(t)|$ is bounded by ${\Upsilon_{sf}}$ and is exponentially convergent to zero in the sense of $|U(t)|\le {\Upsilon_{sf}}e^{-{\lambda_{sf}}t}$
with the positive constants $\lambda_{sf}$ and $\Upsilon_{sf}$ {which only depends on initial values of the system}.
\end{thm}
\begin{proof} Recalling Theorem \ref{main} and Lemma \ref{lem:2}, we have the exponential stability estimates in the sense of the norm $\sum_{i=0}^{2m}\|\partial_x^{i} u(\cdot,t)\|$. Using Sobolev inequality, we obtain the
exponential stability estimate in the sense of the norm $\|u(\cdot,t)\|_{C_{2m-1}}$, which gives the boundedness and exponential
convergence of $U(t)$ by recalling Theorem \ref{main}.
Proof of Theorem \ref{th:ecU} is completed.
\end{proof}

{Note that when we mention the exponential stability result/estimate in the sense of the norm $N_0(t)$, it means there exist positive constants $\bar\Upsilon> 0 $ and $\bar\lambda > 0$ such that $N_0(t)\le \bar\Upsilon e^{-\bar\lambda t}$ where $\bar\Upsilon$ only depends on initial values.}

{\textbf{Brief summary:} The backstepping approach \cite{krstic2008boundary} has been verified as a useful and new method for boundary control of distributed parameter systems. In the proposed method, a PDE backstepping transformation \eqref{eq:t1} is used to convert the original system to the system-$(w(x,t),X(t))$ \eqref{eq:target1}-\eqref{eq:target3} and \eqref{eq:aftt1}, where the state matrix $A+BK$ in \eqref{eq:target1} is Hurwitz and the left boundary condition is $w(0,t)=0$, which are ``stable like'', but the right boundary condition \eqref{eq:aftt1} being a $m$ order ODE-$w(1,t)$ with a number of PDE state perturbations. In order to form an exponentially stable target system, a ODE backstepping transformation \eqref{eq:nct1}-\eqref{eq:nctm} is adopted to convert the ODE states $w(1,t),\cdots,\partial_t^{m-1} w(1,t)$ at the right boundary to $y_1(t),\cdots,y_m(t)$, to build an exponentially stable system-($w(x,t),X(t),y_1(t),\cdots,y_m(t)$) under some control parameters $c_1,\cdots,c_m$ determined by Lyapunov analysis. Through the PDE backstepping and ODE backstepping transformations, the target system and the control law are obtained.

Comparing with a more naive approach, which is to design an intermediate control law for the PDE-ODE system and the intermediate control law to act as a reference to be tracked by the input ODE dynamics with $m$-order relative degree, the merit of the proposed design is avoiding taking $m$ time derivatives of the ``intermediate control law'' and producing high-order-time-derivatives of the state in the control law, especially high-order-time-derivatives of boundary states. }
\section{Output-feedback control design}\label{sec:of}
In Section \ref{eq:sp}, a state-feedback control law at the input ODE is designed to
exponentially stabilize the original ODE-PDE-ODE ``sandwiched'' system. However, the designed
state-feedback control law requires the distributed states $u(x,t)$ in a whole domain,
which are always difficult to obtain in practice. In this section, we propose
an observer-based output feedback control law which
requires only one boundary value as the measurement. An observer is designed
to reconstruct the distributed states $u(x,t)$ and two ODE states $Z(t),X(t)$ using only one boundary measurement $u_x(0, t)$ in Section \ref{sec:observer}. The observer-based output
feedback control law and the stability analysis of the output feedback closed-loop system are presented
in Section \ref{sec:op}.
\subsection{Observer design}\label{sec:observer}
Suppose only one boundary value $u_x(0, t)$ is available for measurement, an observer is designed to
reconstruct the states $u(x,t),Z(t),X(t)$ in this section.

Consider the observer
\begin{align}
&\dot{\hat X}(t) = A\hat X(t) + B{u_x}(0,t)+P_0(u_x(0,t)-\hat u_x(0,t)),\label{eq:ob1}\\
&{\hat u_t}(x,t) = q{\hat u_{xx}}(x,t)+p_1(x)(u_x(0,t)-\hat u_x(0,t)),\label{eq:ob2}\\
&\hat u(0,t) = {C_X}\hat X(t),~\hat u(1,t) = {C_z}\hat Z(t),\label{eq:ob4}\\
&\dot{\hat Z}(t) = {A_z}\hat Z(t) + {B_z}U(t)+P_2(u_x(0,t)-\hat u_x(0,t)),\label{eq:ob5}
\end{align}
where the constant vectors $P_0,P_2$ and the function $p_1(x)$ are to be determined.

Define the observer error as
\begin{align}
\left(\tilde u(x,t),\tilde Z(t), \tilde X(t)\right)=&(u(x,t),Z(t), X(t))\notag\\
&-\left(\hat u(x,t),\hat Z(t),\hat X(t)\right).\label{eq:er}
\end{align}
From \eqref{eq:o1}-\eqref{eq:o5} and \eqref{eq:ob1}-\eqref{eq:ob5}, then the observer error system can be written as
\begin{align}
\dot{\tilde X}(t) &= A\tilde X(t)-P_0\tilde u_x(0,t),\label{eq:e1}\\
{\tilde u_t}(x,t) &= q{\tilde u_{xx}}(x,t)-p_1(x)\tilde u_x(0,t),\label{eq:e2}\\
\tilde u(0,t) &= {C_X}\tilde X(t),~\tilde u(1,t) = {C_z}\tilde Z(t),\label{eq:e4}\\
\dot{\tilde Z}(t) &= {A_z}\tilde Z(t)-P_2\tilde u_x(0,t).\label{eq:e5}
\end{align}
We propose a transformation
\begin{align}
\tilde w(x,t)=\tilde u(x,t)+\vartheta (x)\tilde Z(t) + \theta (x)\tilde X(t),\label{eq:tra1}
\end{align}
where the row vectors $\vartheta (x)$ and $\theta (x)$ are to be determined, to convert the error system \eqref{eq:e1}-\eqref{eq:e5} to the target error system:
\begin{align}
{\tilde w_t}(x,t) &= q{\tilde w_{xx}}(x,t),\label{eq:w2}\\
\tilde w(0,t) &= 0,~~\tilde w(1,t) = 0,\label{eq:w4}\\
\left[ {\begin{array}{*{20}{c}}
{\dot {\tilde Z}(t)}\\
{\dot {\tilde X}(t)}
\end{array}} \right] &= \bigg(\left[ {\begin{array}{*{20}{c}}
{{A_z}}&0\\
0&A
\end{array}} \right] + \left[ {\begin{array}{*{20}{c}}
{{P_2}}\\
{{P_0}}
\end{array}} \right]\left[ {\begin{array}{*{20}{c}}
{\vartheta '(0)}\\{\theta '(0)}
\end{array}} \right]^T\bigg)
\left[ {\begin{array}{*{20}{c}}
{\tilde Z(t)}\\
{\tilde X(t)}
\end{array}} \right] \notag\\
&\quad- \left[ {\begin{array}{*{20}{c}}
{{P_2}}\\
{{P_0}}
\end{array}} \right]{{\tilde w}_x}(0,t).
\label{eq:w5}
\end{align}
By mapping \eqref{eq:e1}-\eqref{eq:e5} and  \eqref{eq:w2}-\eqref{eq:w5}, ${\vartheta (x)},\theta (x)$ should satisfy the following two ODEs:
\begin{align}
\vartheta (x){A_z} - q\vartheta ''(x) = 0,\label{eq:c1}\\
\vartheta (0) = 0,\vartheta (1) =  - {C_z},\label{eq:c2}\\
\theta (x)A - q\theta ''(x)=0,\label{eq:c3}\\
\theta (0) = -{C_X},\theta (1) = 0,\label{eq:c4}
\end{align}
and $p_1(x)$ should be chosen as
\begin{align}
p_1(x)=- \vartheta (x){P_2} - \theta (x){P_0}.\label{eq:p1}
\end{align}
Conditions \eqref{eq:c1}, \eqref{eq:c3}, \eqref{eq:p1} come from achieving \eqref{eq:w2} via \eqref{eq:tra1} from \eqref{eq:e1}-\eqref{eq:e5}. Conditions \eqref{eq:c2} and \eqref{eq:c4} result from \eqref{eq:w4}.

The solution to \eqref{eq:c1}-\eqref{eq:c2} can be represented by
\begin{align}
\vartheta (x)=\left[0,\vartheta' (0)\right]e^{Fx}\left[\begin{array}{c}
                                        I \\
                                        0
                                      \end{array}\right],\label{eq:s1}
\end{align}
with $F=[0,\frac{A_z}{q};I, 0]$ and $I$ being an identity matrix with the appropriate dimension. Especially, for $x = 1$, it holds that
\begin{align}
\vartheta (1)=\left[0,\vartheta' (0)\right]e^{F}\left[\begin{array}{c}
                                        I \\
                                        0
                                      \end{array}\right]= - {C_z}.
\end{align}
According to Lemma 1 in \cite{tang2011state}, when if the matrix $A_z$ has no
eigenvalues of the form $-\bar k^2\pi^2$ for $\bar k\in N$,
\begin{align}
G=[0,I]e^{F}[I,0]^T
\end{align}
is a nonsingular matrix. We then have $\vartheta' (0)=-C_zG^{-1}$.

Therefore the solution \eqref{eq:s1} is
\begin{align}
\vartheta (x)=\left[0,-C_zG^{-1}\right]e^{Fx}\left[\begin{array}{c}
                                        I \\
                                        0
                                      \end{array}\right].\label{eq:ss1}
\end{align}
Similarly, we can obtain the solution of \eqref{eq:c3}-\eqref{eq:c4} as
\begin{align}
\theta (x)=\left[-C_X,C_X\left[
                               \begin{array}{cc}
                                 I & 0 \\
                               \end{array}
                             \right]e^{F_1}\left[\begin{array}{c}
                                        I \\
                                        0
                                      \end{array}\right]
G_1^{-1}\right]e^{F_1x}\left[\begin{array}{c}
                                        I \\
                                        0
                                      \end{array}\right]\label{eq:ss2}
\end{align}
where $G_1=[0,I]e^{F_1}[I,0]^T$ and $F_1=[0,\frac{A}{q};I, 0]$.

Let $P_0,P_2$ to be chosen so that the matrix
\begin{align}
\bar A=A_a+\left[ {\begin{array}{*{20}{c}}
{{P_2}}\\
{{P_0}}
\end{array}} \right]B_a\label{eq:barA}
\end{align}
is Hurwitz,
where
\begin{align}
A_a=\left[ {\begin{array}{*{20}{c}}
{{A_z}}&0\\
0&A
\end{array}} \right],~
B_a=\left[ {\begin{array}{*{20}{c}}
{\vartheta '(0)}\\{\theta '(0)}
\end{array}} \right]^T,
\end{align}
$(A_a,B_a)$ being supposed observable.

Thus, all the quantities needed to implement the observer \eqref{eq:ob1}-\eqref{eq:ob5} are determined. We then give the following theorem which means the observer can effectively track the actual states in the plant \eqref{eq:o1}-\eqref{eq:o5}.
\begin{thm}\label{lem:erm}
Supposing that the matrices $A,A_z$ have no eigenvalues of the form $-\bar k^2\pi^2$, for $\bar k\in N$, consider the observer error system \eqref{eq:e1}-\eqref{eq:e5} obtained from the observer \eqref{eq:ob1}-\eqref{eq:ob5} and the plant \eqref{eq:o1}-\eqref{eq:o5} with initial values $\hat u(x,0)\in H^{{2m}}(0,1)$ and $u(x,0)\in H^{{2m}}(0,1)$.
Then, there exist constants $\Upsilon_{e}>0$ and $\lambda_{e}>0$ such that
\begin{align}
\Theta_{e}(t)\le& \Upsilon_{e}\bigg(\Theta_{e}(0)^2+\tilde u_x(0,0)^2\bigg)^{\frac{1}{2}}e^{-\lambda_{e} t},\label{eq:lemma4norm}
\end{align}
where
\begin{align}
\Theta_{e}(t)=\left(\sum\limits_{i=0}^{2m}\left\|\partial_x^{i}\tilde u(\cdot,t)\right\|^2+\left|\tilde Z(t)\right|^2+\left|\tilde X(t)\right|^2\right)^{\frac{1}{2}}.\label{eq:ernorm}
\end{align}
\end{thm}
\begin{proof}
The proof is shown in the Appendix.
\end{proof}
\subsection{Observer-based output-feedback control law}\label{sec:op}
Replacing the states $u(x,t),X(t)$ in \eqref{eq:Uo} as $\hat u(x,t),\hat X(t)$ defined through the observer \eqref{eq:ob1}-\eqref{eq:ob5}, we obtain the output feedback control law:
\begin{align}
U_{of}(t)&=\mathcal L\hat u(1,t)-\mathcal C\hat u(0,t)-\left(\mathcal L\Phi(1)-\mathcal C\Phi(0)\right)\hat X(t)\notag\\
&\quad-\mathcal L\int_0^1 {{\phi}}(1,y)\hat u(y,t)dy\notag\\
&\quad+\hat F\left(\hat u(0,t),\cdots,\partial _x^{2m-2}\hat u(0,t)\right),\label{eq:Uof}
\end{align}
where $\hat F=\left(\mathcal C \int_0^x {{\phi}}(x,y)\hat u(y,t)dy\right)|_{x=0}$.

Under the proposed output-feedback control law, the closed-loop system which is shown in Fig. \ref{fig:close} is built. The exponential stability results of the closed-loop system are given in the following theorem.
\begin{figure}
\center
\includegraphics[width=8.5cm]{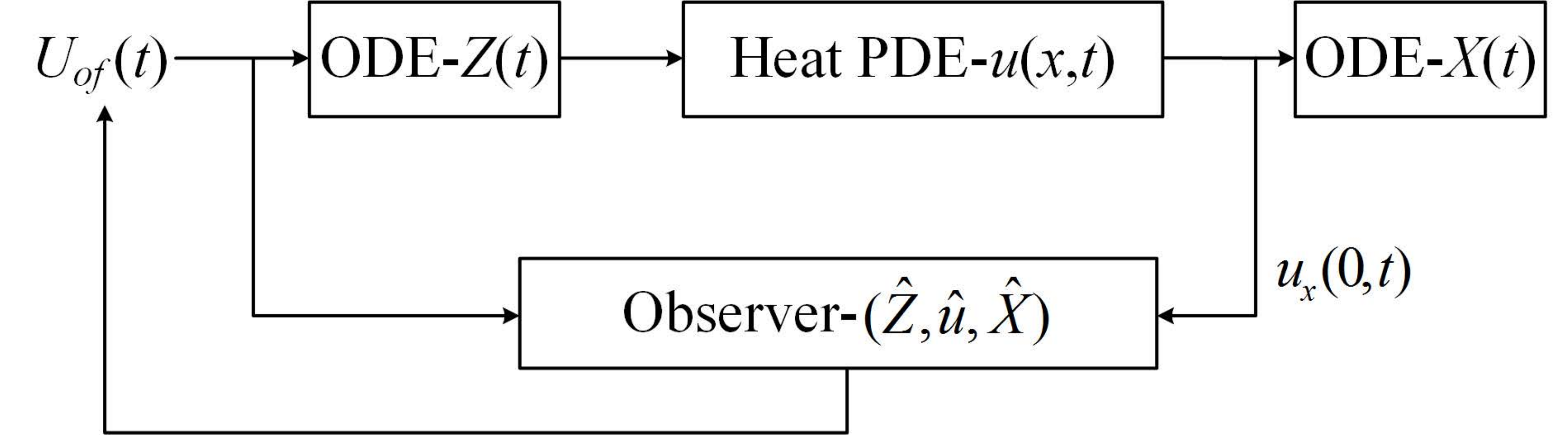}
\caption{The output-feedback closed-loop system consisting of the plant \eqref{eq:o1}-\eqref{eq:o5}, observer \eqref{eq:ob1}-\eqref{eq:ob5} and control input \eqref{eq:Uof}.}
\label{fig:close}
\end{figure}
\begin{thm}\label{th:main1}
Suppose that the matrices $A,A_z$ have no eigenvalues of the form $-\bar k^2\pi^2$, for $\bar k\in N$. For any initial value $(u(x,0),\hat u(x,0))\in H^{{2m}}(0,1)\times H^{{2m}}(0,1)$, the output-feedback closed-loop system consisting of the plant-$(u(x,t),X(t),Z(t))$ \eqref{eq:o1}-\eqref{eq:o5}, the observer-$(\hat u(x,t),\hat X(t),\hat Z(t))$ \eqref{eq:ob1}-\eqref{eq:ob5} and the control input \eqref{eq:Uof} has the following properties:

$i$). There exist positive constants $\Upsilon_{all}$ and $\lambda_{all}$ such that
\begin{align}
\bar\Omega(t)\le \Upsilon_{all}\left(\bar\Omega(0)^2+|\delta(0)|^2+\tilde u_x(0,0)^2\right)^{\frac{1}{2}}e^{-\lambda_{all} t}
\end{align}
where
\begin{align}
&\bar\Omega(t)=\bigg(\|{{u}(\cdot,t)}\|^2+ \|{u_x}(\cdot,t)\|^2+ {\left| {X(t)} \right|^2}+\left| {Z(t)} \right|^2\notag\\
&+\|{{\hat u}(\cdot,t)}\|^2+ \|{\hat u_x}(\cdot,t)\|^2+ {\left| {\hat X(t)} \right|^2}+\left| {\hat Z(t)}\right|^2\bigg)^{\frac{1}{2}},
\end{align}
and $\delta(t)$ will be shown later.

$ii$). The output-feedback control law \eqref{eq:Uof}, $|U_{of}(t)|$ is bounded by ${\Upsilon_{of}}$ and is exponentially convergent to zero in the sense of $|U_{of}(t)|\le {\Upsilon_{of}}e^{-{\lambda_{of}}t}$ with positive constants $\lambda_{of}$ and $\Upsilon_{of}$ which only depends on initial values of the system.
\end{thm}
\begin{proof}
\textbf{Proof of property $i$)}: Inserting the output-feedback control law \eqref{eq:Uof} into \eqref{eq:o5}, adding and subtracting terms, we have
\begin{align}
\dot Z(t) &= {A_z}Z(t) + {B_z}U(t)+B_z\delta(t),\label{eq:dZdel}
\end{align}
where
\begin{align}
\delta(t)=&\mathcal L\tilde u(1,t)-\mathcal C\tilde u(0,t)-\left(\mathcal L\Phi(1)-\mathcal C\Phi(0)\right)\tilde X(t)\notag\\
&-\mathcal L\int_0^1 {{\phi}}(1,y)\tilde u(y,t)dy\notag\\
&+\tilde F\left(\tilde u(0,t),\cdots,\partial _x^{2m-2}\tilde u(0,t)\right)\label{eq:del}
\end{align}
with $\tilde F=\left(\mathcal C \int_0^x {{\phi}}(x,y)\tilde u(y,t)dy\right)|_{x=0}$, and $U(t)$ is in the state-feedback form \eqref{eq:Uo}.

Recalling Theorem \ref{lem:erm}, we have $\delta(t)$ \eqref{eq:del} is exponentially convergent to zero. Together with Theorem \ref{main}, we obtain the exponential stability result as
\begin{align}
&\quad\|u(\cdot,t)\|^2+\|u_x(\cdot,t)\|^2+|X(t)|^2+|Z(t)|^2\notag\\
&\le\|u(\cdot,t)\|^2+\|u_x(\cdot,t)\|^2+|X(t)|^2+|Z(t)|^2+|\delta(t)|^2\notag\\
&\le \Upsilon_{uo}\bigg(\|u(\cdot,0)\|^2+\|u_x(\cdot,0)\|^2+|X(0)|^2\notag\\
&\quad+|Z(0)|^2+|\delta(0)|^2\bigg)e^{-\lambda_{uo}t},
\end{align}
for some positive $\Upsilon_{uo},\lambda_{uo}$.

Recalling Theorem \ref{lem:erm} and \eqref{eq:er}, we obtain
\begin{align}
&\|\hat u(\cdot,t)\|^2+\|\hat u_x(\cdot,t)\|^2+|\hat X(t)|^2+|\hat Z(t)|^2\notag\\
\le& \hat\xi_0\bigg(\|u(\cdot,t)\|^2+\|u_x(\cdot,t)\|^2+|X(t)|^2+|Z(t)|^2\notag\\
&+\|\tilde u(\cdot,t)\|^2+\|\tilde u_x(\cdot,t)\|^2+|\tilde X(t)|^2+|\tilde Z(t)|^2\bigg)\notag\\
\le& \Upsilon_{\hat u}\bigg(\|u(\cdot,0)\|^2+\|u_x(\cdot,0)\|^2+|X(0)|^2+|Z(0)|^2+|\delta(0)|^2\notag\\
&+\|\tilde u(\cdot,0)\|^2+\|\tilde u_x(\cdot,0)\|^2+|\tilde X(0)|^2+|\tilde Z(0)|^2\notag\\
&+\tilde u_x(0,0)^2\bigg)e^{-\lambda_{\hat u}t},
\end{align}
for some positive $\hat\xi_0, \Upsilon_{\hat u},\lambda_{\hat u}$. With representing the initial values $\|\tilde u(\cdot,0)\|^2+\|\tilde u_x(\cdot,0)\|^2+|\tilde X(0)|^2+|\tilde Z(0)|^2$ as $\|u(\cdot,0)\|^2+\|u_x(\cdot,0)\|^2+|X(0)|^2+|Z(0)|^2$+$\|\hat u(\cdot,0)\|^2+\|\hat u_x(\cdot,0)\|^2+|\hat X(0)|^2+|\hat Z(0)|^2$ by recalling \eqref{eq:er} and using Cauchy-Schwarz inequality, Property $i$) can then be proved.

\textbf{Proof of property $ii$)}: Recalling the exponential stability estimates $\sum_{i=0}^{2m}\|\partial_{x}^{i}u(\cdot,t)\|$ and $\sum_{i=0}^{2m}\|\partial_{x}^{i}\tilde u(\cdot,t)\|$ proved by Theorem \ref{main}, Lemma \ref{lem:2} and Theorem \ref{lem:erm}, we have the exponential stability estimate in the sense of the norm $\sum_{i=0}^{2m}\|\partial_{x}^{i}\hat u(\cdot,t)\|$. Using Sobolev inequality, we obtain the
exponential stability estimate in term of the norm $\|\hat u(\cdot,t)\|_{C_{2m-1}}$, together with the property $i$), which give the exponential
convergence of $U_{of}(t)$ which uses the signals $\sum_{i=0}^{2m-1}\partial_x^{i}\hat u(1,t)$, $\sum_{i=0}^{2m-1}\partial_x^{i}\hat u(0,t)$, $\hat X(t)$ and $\hat u(x,t)$.

The proof of Theorem \ref{th:main1} is completed.
\end{proof}
\section{Simulation}\label{sec:sim}
Consider the simulation example where the plant coefficients in \eqref{eq:o1}-\eqref{eq:o5} are $A=[1,1;1,0.5]$, $A_z=[0,1;1,1]$, $B_z=B=[0,1]^T$, $C_X=C_z=[1,0]^T$ and $q=1$. {Two ODEs sandwiching the heat PDE are considered as two-order systems here, i.e., $m=2$ in the above design and analysis, because the second-order ODE is a classic system which can describe many actuator and sensor dynamics.} The simulation is conducted based on the finite difference
method with dividing the spatial and time domains into a grid as $x_0,\cdots,x_f$ and $t_0,\cdots,t_{n^*}$ respectively, where the time step and space step sizes are $0.001$ and $0.05$. The initial conditions of the plant are defined as $u(x,0)=\sin(2\pi x)$, $X(0)=[x_1(0),x_2(0)]^T=[u(0,0),0]^T$, $Z(0)=[z_1(0),z_2(0)]^T=[u(1,0),0]^T$. The initial conditions of the observer \eqref{eq:ob1}-\eqref{eq:ob5} are $\hat u(x,0)=0$, $\hat X(0)=\hat Z(0)=[0,0]^T$. Choose the control parameters $c_1=c_2=3$, $K=[-10,-5]$, $P_0=[-2,-4]^T$ and $P_2=[-4,-12]^T$. Apply the output-feedback control law \eqref{eq:Uof} with $m=2$, which is constructed by the states $\sum_{i=0}^{3}\partial_x^{i}\hat u(1,t)$, $\sum_{i=0}^{3}\partial_x^{i}\hat u(0,t)$, $\hat X(t)$ and $\hat u(x,t)$ of the observer \eqref{eq:ob1}-\eqref{eq:ob5} built using the measurement $u_x(0,t)$, into the plant \eqref{eq:o1}-\eqref{eq:o5}. Note that the third-order spatial derivatives are determined by finite difference method such as
\begin{align}
&\quad\hat u_{xxx}(1,t_j)\notag\\
&=\frac{\hat u(x_f,t_j)-3\hat u(x_{f-1},t_j)+3\hat u(x_{f-2},t_j)-\hat u(x_{f-3},t_j)}{\Delta h^3},
\end{align}
where $\Delta h$ is spatial step size and $t_j$ is the current time point. $\hat u_{xxx}(1,t_j)$ is used to determine $U(t_{j+1})$.

The responses of the output-feedback closed-loop system are shown following.
\begin{figure}
\begin{minipage}{0.495\linewidth}
  \centerline{\includegraphics[width=4.9cm]{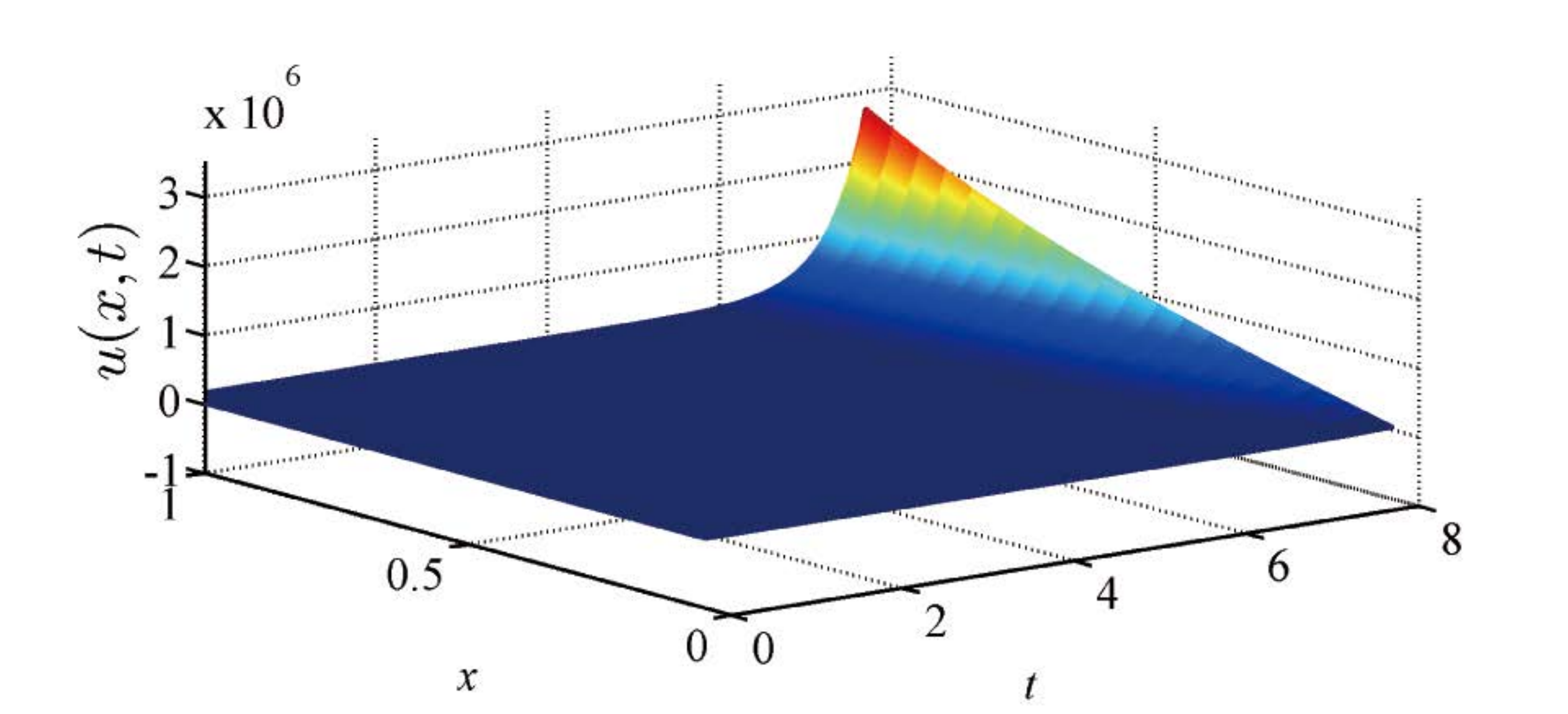}}
  \centerline{(a) uncontrolled case.}
\end{minipage}
\begin{minipage}{.495\linewidth}
  \centerline{\includegraphics[width=4.9cm]{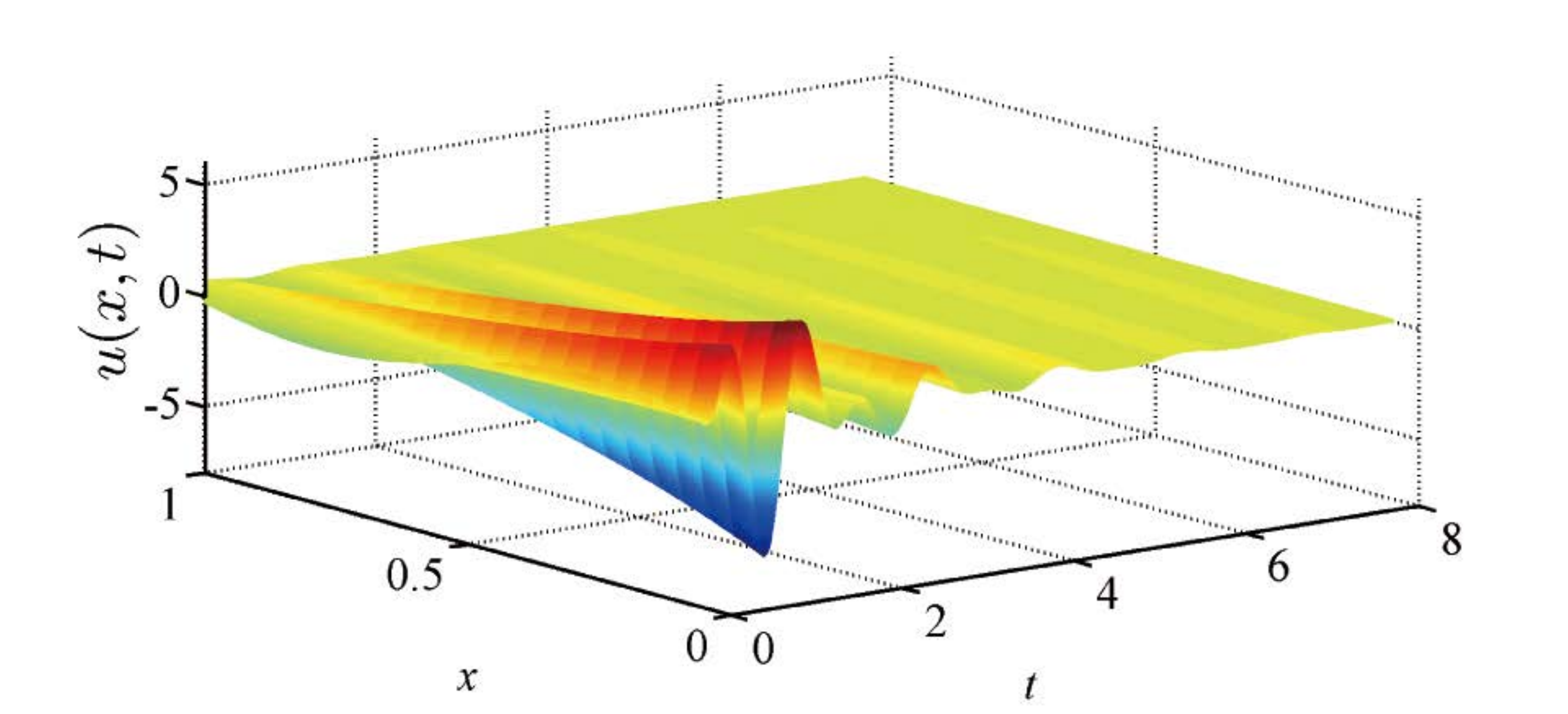}}
  \centerline{(b) controlled case.}
\end{minipage}
\caption{Responses of the heat PDE states $u(x,t)$ in the uncontrolled case and under the output-feedback control law \eqref{eq:Uof}.}
\label{fig:u}
\end{figure}
\begin{figure}
\begin{minipage}{0.49\linewidth}
  \centerline{\includegraphics[width=4.7cm]{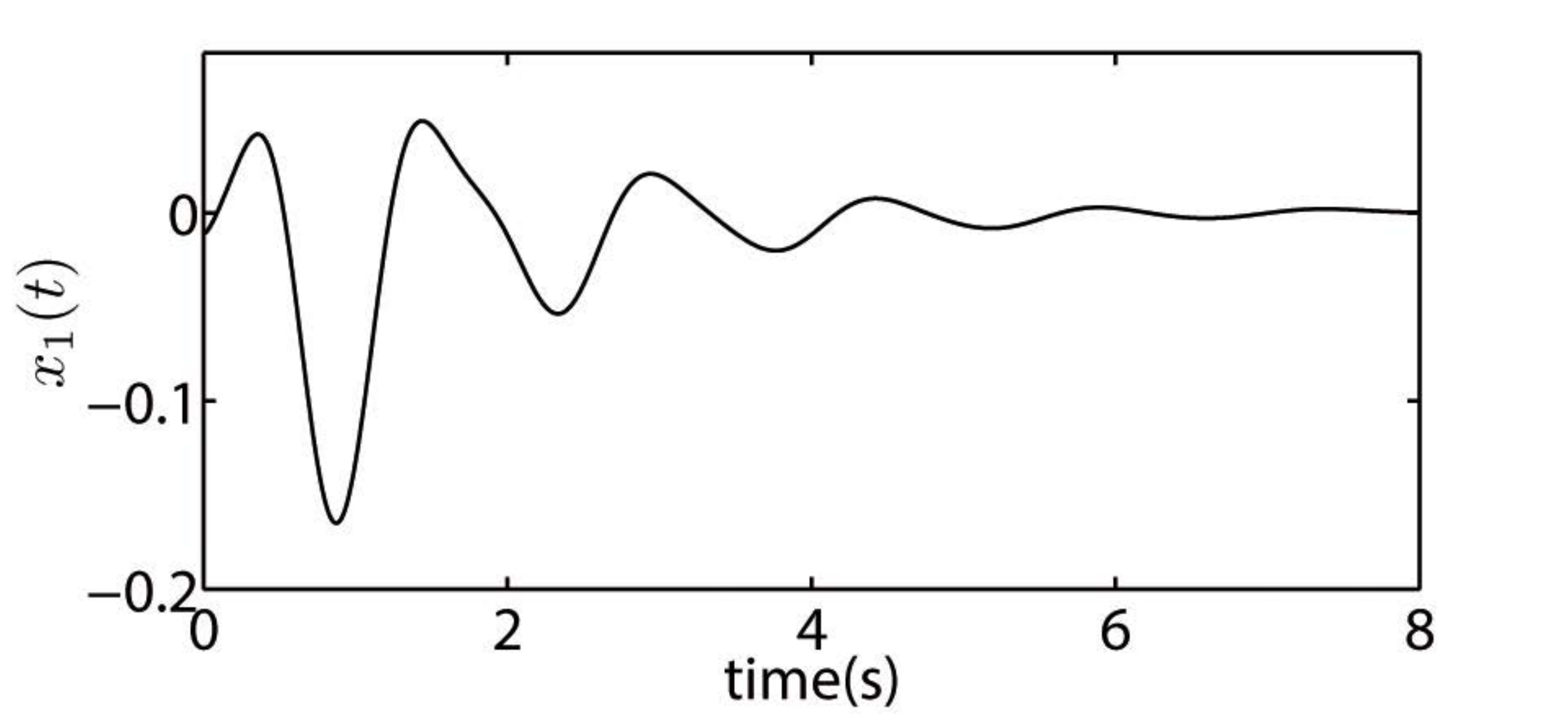}}
  \centerline{(a) $x_1(t)$.}
\end{minipage}
\begin{minipage}{.49\linewidth}
  \centerline{\includegraphics[width=4.7cm]{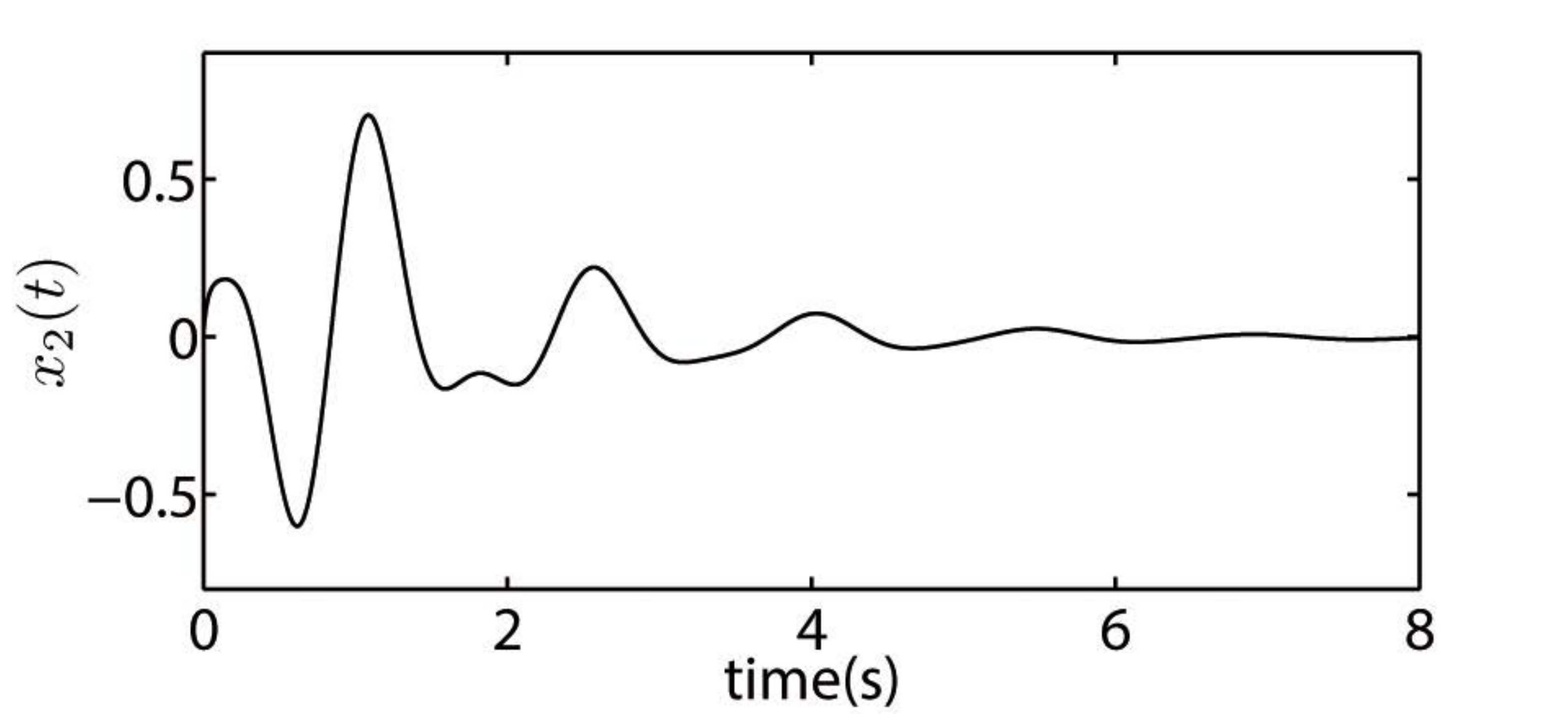}}
  \centerline{(b) $x_2(t)$.}
\end{minipage}
\caption{Responses of the ODE state $X(t)$ under the output-feedback control law \eqref{eq:Uof}.}
\label{fig:X}
\end{figure}
\begin{figure}[!ht]
\begin{minipage}{0.49\linewidth}
  \centerline{\includegraphics[width=4.7cm]{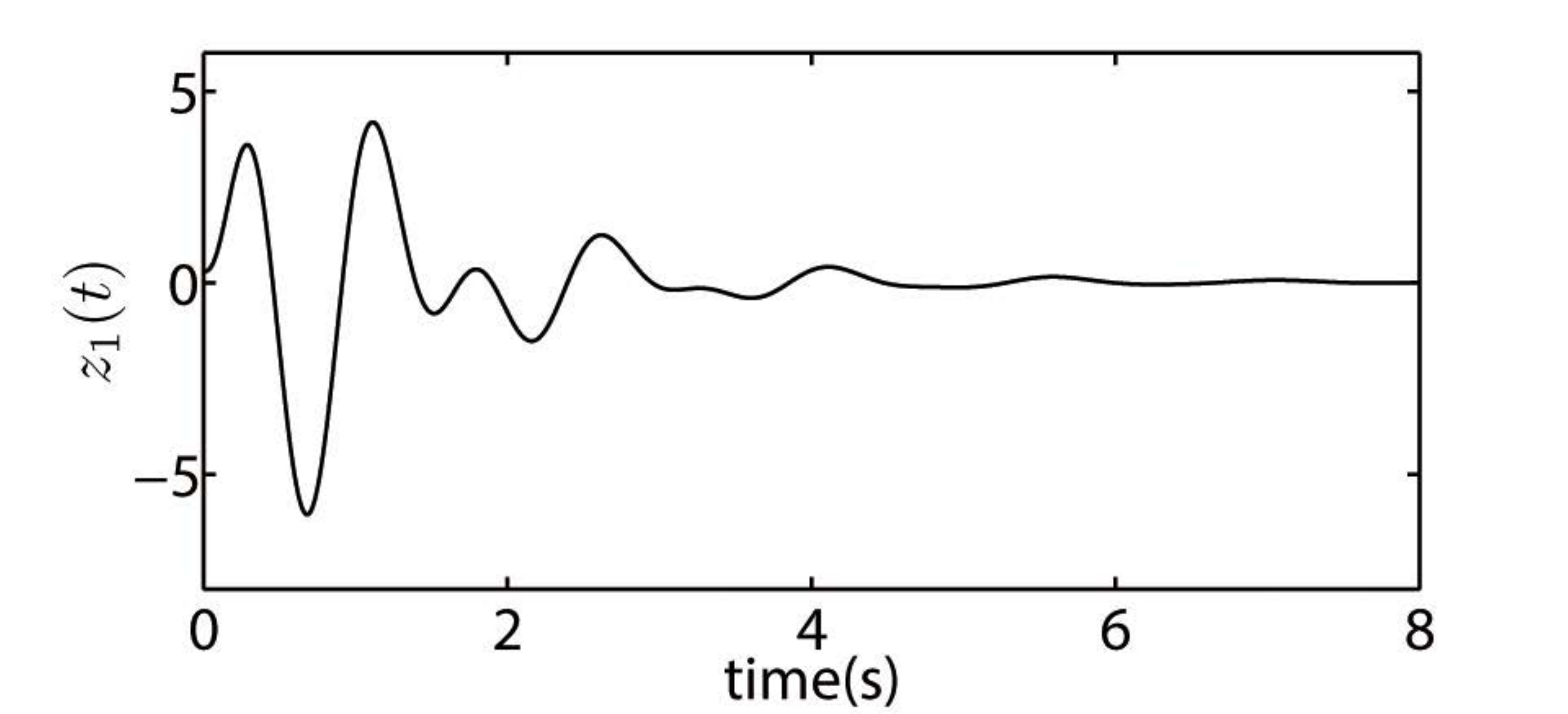}}
  \centerline{(a) $z_1(t)$.}
\end{minipage}
\begin{minipage}{.49\linewidth}
  \centerline{\includegraphics[width=4.7cm]{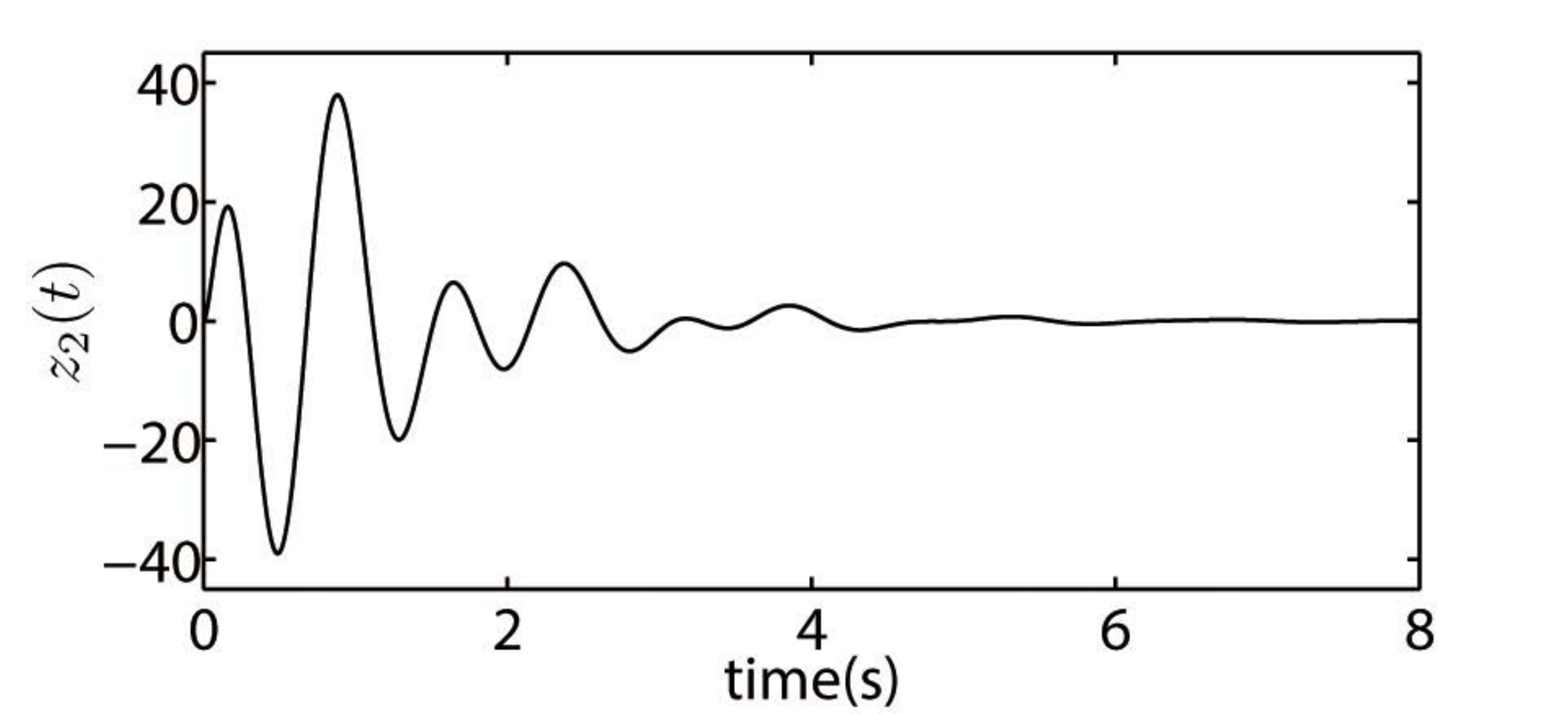}}
  \centerline{(b) $z_2(t)$.}
\end{minipage}
\caption{Responses of the ODE state $Z(t)$ under the output-feedback control law \eqref{eq:Uof}.}
\label{fig:Z}
\end{figure}
\begin{figure}[!ht]
\begin{minipage}{0.49\linewidth}
  \centerline{\includegraphics[width=4.9cm]{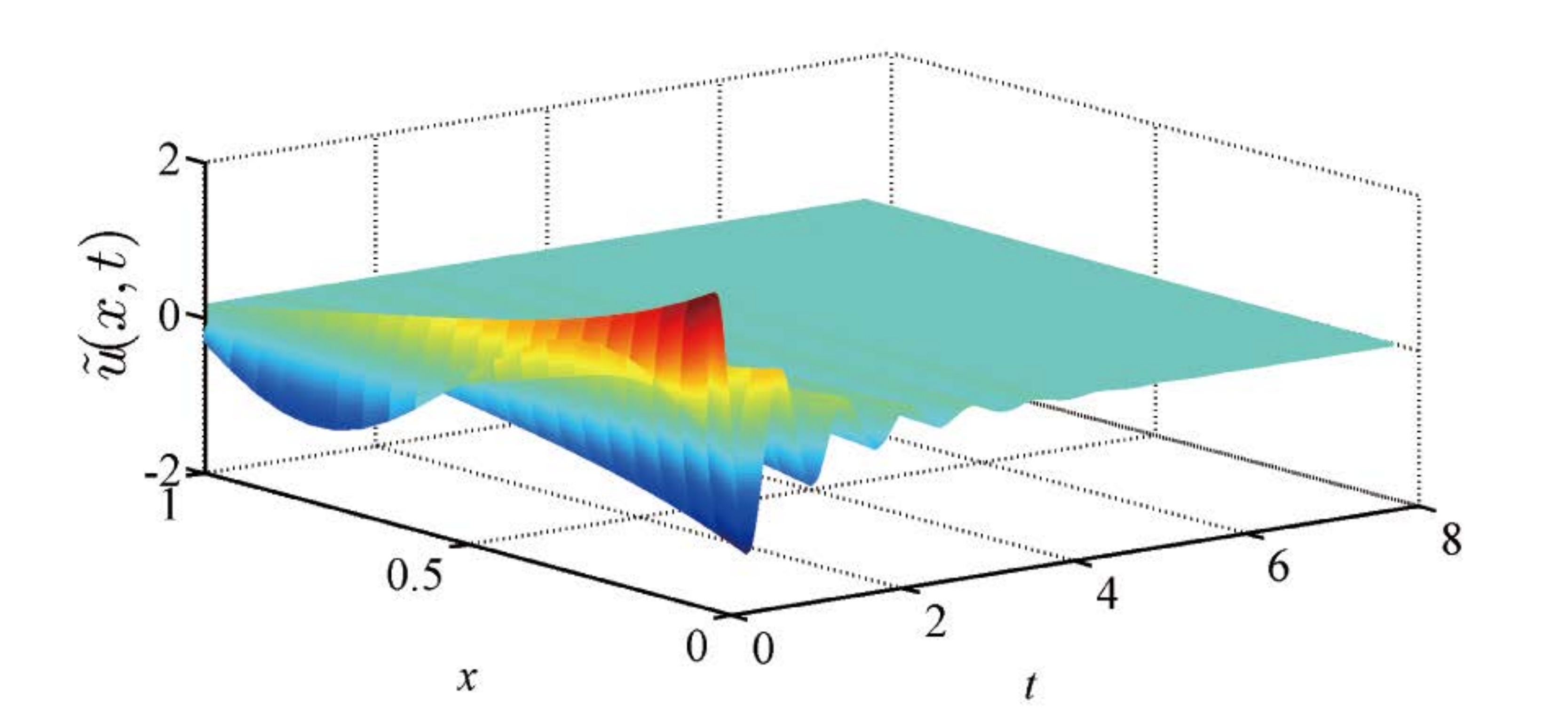}}
  \centerline{(a) observer error.}
\end{minipage}
\begin{minipage}{.49\linewidth}
  \centerline{\includegraphics[width=4.7cm]{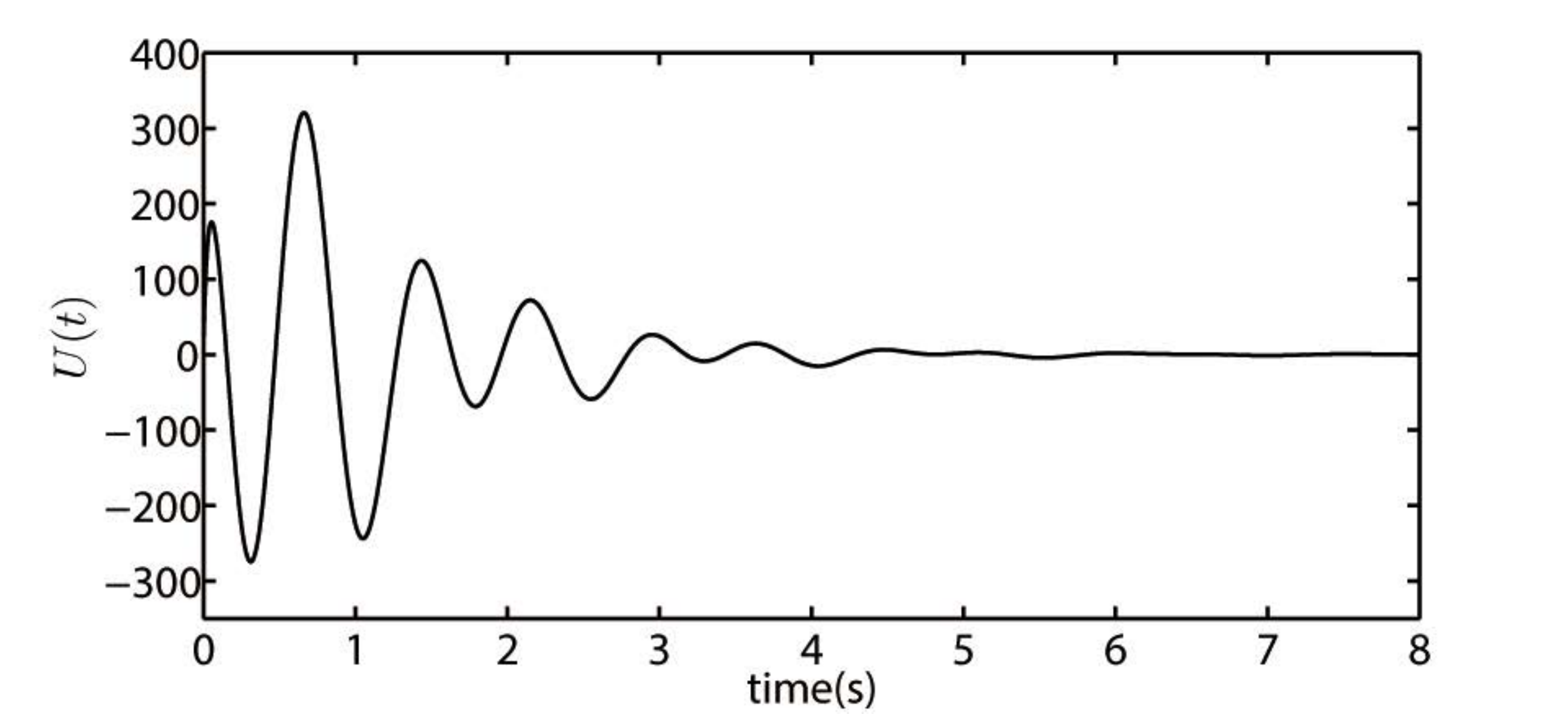}}
  \centerline{(b) control input.}
\end{minipage}
\caption{a) Observer errors $\tilde u=u-\hat u$ of the observer \eqref{eq:ob1}-\eqref{eq:ob5}. b) The output-feedback control input \eqref{eq:Uof}.}
\label{fig:f}
\end{figure}

As Fig.\ref{fig:u} shows, the response $u(x, t)$ of the heat PDE exhibits
unstable behaviour in the uncontrolled case while
the convergent manner of the response $u(x,t)$ is achieved
when we apply the proposed output feedback control law \eqref{eq:Uof}. Similarly, Figs. \ref{fig:X}-\ref{fig:Z} show the
ODE states $X(t)=[x_1(t),x_2(t)]^T$ and $Z(t) = [z_1(t),z_2(t)]^T$ are also convergent to zero in the output-feedback closed-loop system.
It can be seen in Fig. \ref{fig:f} a) that the observer errors $\tilde u(x,t)=u(x,t)-\hat u(x,t)$ also converge fast to zero in the closed-loop system. Moreover, Fig. \ref{fig:f} b) shows the output-feedback control input are bounded and convergent to zero.
\section{Conclusion and Future work}\label{sec:con}
In this paper, we present a methodology combining PDE backstepping and ODE backstepping to stabilize
a parabolic PDE sandwiched between two arbitrary-order ODEs. An observer is also designed only using one PDE boundary value $u_x(0,t)$ to reconstruct all PDE and ODE states. The observer-based output-feedback control law is proposed and the exponential stability of the closed-loop system is proved via Lyapunov analysis. Moreover, the boundedness and exponential convergence of the designed control input is also proved in this
paper. These theoretical results are verified via the simulation as well. {In the future work, more general ODE dynamics in the input channel will be considered in the control design}.
\section{Appendix}
\textbf{\emph{Proof of Lemma \ref{lem:2}:}}

\subsubsection{Proof of i)}
Taking 2 and 3 times derivative of \eqref{eq:target2} with
respect to $x$ respectively, we have
\begin{align}
w_{txx}(x,t)=q\partial_x^{4}w(x,t),\label{eq:2x}\\
w_{txxx}(x,t)=q\partial_x^{5}w(x,t).\label{eq:3x}
\end{align}

Define a Lyapunov function
\begin{align}
V_2(t)=R_0V(t)+\frac{a_2}{2}\|\partial_x^{2}w(\cdot,t)\|^2+\frac{a_3}{2}\|\partial_x^{3}w(\cdot,t)\|^2,\label{eq:V2}
\end{align}
where $a_2,a_3$ are positive constants to be determined later.

Defining
\begin{align}
\Omega_1(t)=\Omega(t)+\|\partial_x^{2}w(\cdot,t)\|^2+\|\partial_x^{3}w(\cdot,t)\|^2.\label{eq:omega1}
\end{align}
We have
\begin{align}
\theta_{11}\Omega_1(t)\le V_2(t)\le \theta_{12}\Omega_1(t),\label{eq:V2bound}
\end{align}
for some positive $\theta_{11},\theta_{12}$.

Taking the derivative of \eqref{eq:V2} along \eqref{eq:2x}-\eqref{eq:3x} and \eqref{eq:target3}, recalling \eqref{eq:Vfinal}, applying Agmon's inequality, Young's inequality and
Cauchy-Schwarz inequality, yields
\begin{align}
&\dot V_2(t)\notag\\
\le& -R_0\lambda V(t) - \left(\frac{{{a_3}q}}{2} - \frac{1}{{4{r_{2}}}}{a_2} - \frac{1}{{4{r_{3}}}}\frac{{a_3}}{q}\right){w_{xxx}}(0,t)^2\notag\\
 &- \left(({a_2} - {a_3})q - \frac{1}{{4{r_{2}}}}{a_2} - \frac{1}{{4{r_{3}}}}\frac{{a_3}}{q}\right){\left\| {{w_{xxx}}} \right\|^2}\notag\\
 &- \left(\frac{1}{2}{a_3}q - \frac{1}{{4{r_{2}}}}{a_2} - \frac{1}{{4{r_{3}}}}\frac{{a_3}}{q}\right){\left\| {{w_{xxxx}}} \right\|^2} \notag\\
 &+ {r_{2}}{a_2}{w_t}{(1,t)^2} + {r_{3}}\frac{{a_3}}{q}{w_{tt}}{(1,t)^2}-R_0\bar\xi_aw_x(0,t)^2,\label{eq:dv2f0}
\end{align}
where $r_2,r_3$ are positive constants from Young's inequality to be chosen later.

Recalling \eqref{eq:nct1}-\eqref{eq:nctm} and \eqref{eq:norm2}-\eqref{eq:tVt}, considering the case of $m\ge3$, we have
\begin{align}
\dot V_2(t)\le&-\frac{R_0}{2}\lambda V(t)\notag\\
&  - \left(\frac{{{a_3}q}}{2} - \frac{1}{{4{r_{2}}}}{a_2} - \frac{1}{{4{r_{3}}}}\frac{{a_3}}{q}\right){w_{xxx}}(0,t)^2\notag\\
 &- \left(({a_2} - {a_3})q - \frac{1}{{4{r_{2}}}}{a_2} - \frac{1}{{4{r_{3}}}}\frac{{a_3}}{q}\right){\left\| {{w_{xxx}}} \right\|^2}\notag\\
 &- \left(\frac{1}{2}{a_3}q - \frac{1}{{4{r_{2}}}}{a_2} - \frac{1}{{4{r_{3}}}}\frac{{a_3}}{q}\right){\left\| {{w_{xxxx}}} \right\|^2} \notag\\
 &-\left(\frac{R_0}{2}\lambda\theta_{1}-2{r_{2}}{a_2}c_1^2-3{r_{3}}\frac{1}{q}{a_3}(c_1^2-1)^2\right)y_1(t)^2\notag\\
 & -\left(\frac{R_0}{2}\lambda\theta_{1}-3{r_{3}}\frac{1}{q}{a_3}(c_1+c_2)^2-2{r_{2}}{a_2}\right)y_2(t)^2\notag\\
 &-\left(\frac{R_0}{2}\lambda\theta_{1}-3{r_{3}}\frac{1}{q}{a_3}\right)y_3(t)^2.\label{eq:dv2f}
\end{align}
Choosing $a_{2}>a_{3}$ and sufficiently large ${r_{2}}$, ${r_{3}}$, $R_0$, we arrive at
\begin{align}
\dot V_2(t)\le&-\frac{R_0}{2}\lambda V(t)\notag\\
&- \frac{1}{5}\left(({a_2} - {a_3})q - \frac{1}{{4{r_{2}}}}{a_2} - \frac{1}{{4{r_{3}}}}\frac{{a_3}}{q}\right)\notag\\
&\quad\times\left({\left\| {{w_{xx}}} \right\|^2}+{\left\| {{w_{xxx}}} \right\|^2}\right),
\end{align}
where Poincar$\acute e$ inequality is used.
Therefore, we obtain
\begin{align}
\dot V_{2}(t)\le- \lambda_{2} V_{2}(t)\label{eq:dv2final}
\end{align}
for some positive $\lambda_{2}$.
Recalling \eqref{eq:V2bound} and invertibility of the backstepping transformations, we obtain $i)$ in Lemma \ref{lem:2}.

In the case of $m=2$: \eqref{eq:wtm} and Cauchy-Schwarz inequality can be used to rewrite $\frac{{r_{3}}{a_3}}{q}w_{tt}(1,t)^2$ in \eqref{eq:dv2f} as $\frac{{r_{3}}a_3}{q}{}(\xi_{b}w(1,t)^2$+$\xi_{b}w_t(1,t)^2$+$\xi_{b}|X(t)|^2$+$\xi_{b}\|w\|^2)$ with some positive $\xi_{b}$, where $-\frac{R_0}{2}\lambda\theta_{1}(y_1(t)^2+y_2(t)^2)$ and $-\frac{R_0}{2}\lambda\theta_{1}(|X(t)|^2+\|w\|^2)$ derived from $-\frac{R_0}{2}\lambda V(t)$ with large enough $R_0$ can ``cancel'' $\frac{{r_{3}}a_3}{q}\xi_{b}w(1,t)^2$, $(\frac{{r_{3}}a_3}{q}\xi_{b}+r_2a_2)w_t(1,t)^2$ according to \eqref{eq:nct1}-\eqref{eq:nct2}, and $\frac{{r_{3}}a_3}{q}\xi_{b}|X(t)|^2$, $\frac{{r_{3}}a_3}{q}\xi_{b}\|w\|^2$ respectively. \eqref{eq:dv2final} can then also be obtained.

In the case of $m=1$: substituting $m=1$ into \eqref{eq:wtm} and taking the time derivative, applying Cauchy-Schwarz inequality and recalling \eqref{eq:wtm0}, we have
\begin{align}
w_{tt}(1,t)^2&= \xi_c\|w_{xx}\|^2+ \xi_c|X(t)|^2+\xi_cw_x(0,t)^2\notag\\
  &\quad+\xi_cw(1,t)^2+\xi_c\|w\|^2.\label{eq:wtm1a}
\end{align}
Substituting \eqref{eq:wtm0} and \eqref{eq:wtm1a} with \eqref{eq:nct1} into \eqref{eq:dv2f0}, we can obtain
\begin{align}
&\dot V_2(t)\le-\frac{R_0}{2}\lambda V(t)-\left(R_0\bar\xi_a-\frac{{r_{3}}{a_3}\xi_c}{q}\right)w_x(0,t)^2\notag\\
&  - \left(\frac{{{a_3}q}}{2} - \frac{1}{{4{r_{2}}}}{a_2} - \frac{1}{{4{r_{3}}}}\frac{{a_3}}{q}\right){w_{xxx}}(0,t)^2\notag\\
 &- \left(({a_2} - {a_3})q - \frac{1}{{4{r_{2}}}}{a_2} - \frac{1}{{4{r_{3}}}}\frac{{a_3}}{q}\right){\left\| {{w_{xxx}}} \right\|^2}\notag\\
 &- \left(\frac{1}{2}{a_3}q - \frac{1}{{4{r_{2}}}}{a_2} - \frac{1}{{4{r_{3}}}}\frac{{a_3}}{q}\right){\left\| {{w_{xxxx}}} \right\|^2} \notag\\
 &-\left(\frac{R_0}{2}\lambda\theta_{1}-{r_{2}}{a_2}\xi_a-\frac{{r_{3}}{a_3}\xi_c}{q}\right)y_1(t)^2\notag\\
 &-\left(\frac{R_0}{2}\lambda\theta_{1}-{r_{2}}{a_2}\xi_a-\frac{{r_{3}}{a_3}\xi_c}{q}\right)|X(t)|^2\notag\\
 &-\left(\frac{R_0}{2}\lambda\theta_{1}-{r_{2}}{a_2}\xi_a-\frac{{r_{3}}{a_3}\xi_c}{q}\right)\|w\|^2+\frac{{r_{3}}{a_3}\xi_c}{q}\|w_{xx}\|^2.
\end{align}
Choosing sufficiently large $R_0$ and using Poincar$\acute e$ inequality, we have
\begin{align}
&\dot V_2(t)\le-\frac{R_0}{2}\lambda V(t)\notag\\
&  - \left(\frac{{{a_3}q}}{2} - \frac{1}{{4{r_{2}}}}{a_2} - \frac{1}{{4{r_{3}}}}\frac{{a_3}}{q}\right)({w_{xxx}}(0,t)^2+\|w_{xxxx}\|^2)\notag\\
 &- \left(\frac{1}{5}({a_2} - {a_3})q - \frac{1}{{20{r_{2}}}}{a_2} - \frac{1}{{20{r_{3}}}}\frac{{a_3}}{q}\right){\left\| {{w_{xxx}}} \right\|^2}\notag\\
 &- \left(\frac{1}{5}({a_2} - {a_3})q - \frac{1}{{20{r_{2}}}}{a_2} - \frac{1}{{20{r_{3}}}}\frac{{a_3}}{q}-\frac{{r_{3}}{a_3}\xi_c}{q}\right)\|w_{xx}\|^2.
\end{align}
In order to make the coefficients before $w_{xxx}(0,t)^2+\|w_{xxxx}\|^2$, $\|w_{xxx}\|^2$ and $\|w_{xx}\|^2$ positive, $a_2, a_3$ and $r_2, r_3$ should be chosen satisfying
\begin{align}
\frac{\frac{q}{5}+\frac{1}{20r_3q}+\frac{r_3\xi_c}{q}}{\frac{q}{5}-\frac{1}{20r_2}}a_3<a_2<4r_2\left(\frac{q}{2}-\frac{1}{4qr_3}\right)a_3,\\
r_3>\frac{1}{2q^2},~~
r_2>\frac{\frac{1}{4}+\frac{1}{16r_3q^2}+\frac{5r_3\xi_c}{4q^2}}{\frac{q}{2}-\frac{1}{4qr_3}}+\frac{1}{4q}.
\end{align}
\eqref{eq:dv2final} is thus obtained as well in the case of $m=1$.

Therefore, \eqref{eq:dv2final} can be obtained when $m\ge 1$. The proof of the first result in Lemma \ref{lem:2} is completed.

\subsubsection{Proof of ii)}
Through similar processes from \eqref{eq:2x} to \eqref{eq:dv2final}, we can obtain the exponential stability estimates in the sense of $(\|\partial_x^{4}w(\cdot,t)\|^2+\|\partial_x^{5}w(\cdot,t)\|^2)^{\frac{1}{2}}$, $(\|\partial_x^{6}w(\cdot,t)\|^2+\|\partial_x^{7}w(\cdot,t)\|^2)^{\frac{1}{2}}$ and so on, i.e., the exponential stability estimates in the sense of $(\|\partial_x^{4}u(\cdot,t)\|^2+\|\partial_x^{5}u(\cdot,t)\|^2)^{\frac{1}{2}}$, $(\|\partial_x^{6}u(\cdot,t)\|^2+\|\partial_x^{7}u(\cdot,t)\|^2)^{\frac{1}{2}}$ and so on, because of the invertibility of the backstepping transformation.

\textbf{Next, we prove the exponential stability estimate in the sense of $(\|\partial_x^{2m-2}u(\cdot,t)\|^2+\|\partial_x^{2m-1}u(\cdot,t)\|^2)^{\frac{1}{2}}$.}

Taking $2m-2$ and $2m-1$ times derivative of \eqref{eq:target2} with respect to $x$ respectively, we have
\begin{align}
\partial_x^{2m-2}w_t(x,t)&=q\partial_x^{2m}w(x,t),\label{eq:2m-2}\\
\partial_x^{2m-1}w_t(x,t)&=q\partial_x^{2m+1}w(x,t).\label{eq:2m-1}
\end{align}
Apply a Lyapunov function
\begin{align}
{V_{{\rm{2m-2}}}}(t) = \frac{{{a_{2m-2}}}}{2}\| {\partial _x^{2m-2}{w}} (\cdot,t)\|^2+ \frac{{{a_{2m -1}}}}{2}\| {\partial _x^{2m-1}{w}} (\cdot,t)\|^2.\label{eq:V2m2}
\end{align}
Taking the derivative of \eqref{eq:V2m2}, we have
\begin{align*}
&{{\dot V}_{2m - 2}}(t)\notag\\
\le&  - \bigg(\frac{{{a_{2m - 1}}q}}{2} - \frac{{a_{2m - 2}}}{{4q^{m-2}{r_{2m - 2}}}}\notag\\
&\quad-\frac{{a_{2m - 1}}}{{4q^{m-1}{r_{2m - 1}}}}\bigg){\partial _x^{2m - 1}w(0,t)}^2\\
 &- \bigg(({a_{2m - 2}} - {a_{2m - 1}})q- \frac{{a_{2m - 2}}}{{4q^{m-2}{r_{2m - 2}}}}\notag\\
  &\quad- \frac{{a_{2m - 1}}}{{4q^{m-1}{r_{2m - 1}}}}\bigg){\left\| {\partial _x^{2m - 1}w} \right\|^2}\\
 &- \bigg(\frac{{a_{2m - 1}}q}{2} - \frac{{a_{2m - 2}}}{{4q^{m-2}{r_{2m - 2}}}}- \frac{{a_{2m - 1}}}{{4q^{m-1}{r_{2m - 1}}}}\bigg){\left\| {\partial _x^{2m}w} \right\|^2}\\
 &+ \frac{{r_{2m - 2}}{a_{2m - 2}}}{{{q^{m - 2}}}}{\partial _t^{m - 1}{w}(1,t)}^2+ \frac{{r_{2m - 1}}{a_{2m - 1}}}{{{q^{m - 1}}}}{\partial _t^mw(1,t)}^2.
\end{align*}
Note that ${\partial _x^{2m - 1}w(0,t)}^2\triangleq (\partial _x^{2m - 1}w(x,t)|_{x=0})^2$ and ${\partial _t^mw(1,t)}^2\triangleq(\partial _t^{m}w(x,t)|_{x=1})^2$.

We choose $a_{2m-2}>a_{2m-1}$ and sufficiently large ${r_{2m-2}}$, ${r_{2m-1}}$ to make
\begin{align}
&\quad\dot V_{2m-2}(t)\notag\\
&\le- \frac{1}{5}\bigg(({a_{2m - 2}} - {a_{2m - 1}})q- \frac{{a_{2m - 2}}}{{4q^{m-2}{r_{2m - 2}}}}\notag\\
  &\quad-\frac{{a_{2m - 1}}}{{4q^{m-1}{r_{2m - 1}}}}\bigg)\left({\left\| {\partial _x^{2m - 1}w} \right\|^2}+\left\| {\partial _x^{2m - 2}w} \right\|^2\right)\notag\\
 &\quad+ \frac{{r_{2m - 2}}{a_{2m - 2}}}{{{q^{m - 2}}}}{\partial _t^{m - 1}{w}(1,t)}^2+ \frac{{r_{2m - 1}}{a_{2m - 1}}}{{{q^{m - 1}}}}{\partial _t^mw(1,t)}^2\notag\\
&\le -\bar\lambda_{2m-2}V_{2m-2}(t)+ \frac{{r_{2m - 2}}{a_{2m - 2}}}{{{q^{m - 2}}}}{\partial _t^{m - 1}{w}(1,t)}^2\notag\\
&\quad+ \frac{{r_{2m - 1}}{a_{2m - 1}}}{{{q^{m - 1}}}}{\partial _t^mw(1,t)}^2,\label{eq:dv2m2}
\end{align}
for some positive $\lambda_{2m-2}$, where Poincar$\acute e$ inequality is used.

Applying Cauchy-Schwarz inequality into \eqref{eq:wtm}, we have
\begin{align}
\left|\partial _t^mw(1,t)\right|^2\le \xi\|w\|^2+\xi|X(t)|^2+\xi\sum_{i=0}^{m-1}\left|\partial _t^iw(1,t)\right|^2\label{eq:dwm}
\end{align}
with some positive $\xi$. By recalling \eqref{eq:nct1}-\eqref{eq:nctm}, \eqref{eq:norm2}-\eqref{eq:tVt} and \eqref{eq:Vfinal}, we obtain the exponential convergence of $\partial _t^mw(1,t)$ via \eqref{eq:dwm}, i.e.,
\begin{align}
\left|\partial _t^mw(1,t)\right|\le \bar\Upsilon_me^{-\xi_1t}=\eta_m(t),\label{eq:etam}
\end{align}
for some positive $\bar\Upsilon_m$ and $\xi_1$.

Define a Lyapunov function
\begin{align}
{V_{u1}}(t) &= V_{2m-2}(t)+\bar R_0V(t)+\frac{\bar R_1}{2}\eta_m(t)^2,\label{eq:Vmn1}
\end{align}
where $\bar R_0,\bar R_1$ are positive constants to be determined later.

Taking the derivative of \eqref{eq:Vmn1} and recalling \eqref{eq:Vfinal}, \eqref{eq:dv2m2}, \eqref{eq:etam}, we have
\begin{align}
&\dot V_{u1}(t) \le -\bar\lambda_{2m-2}V_{2m-2}(t)+ \frac{{r_{2m - 2}}{a_{2m - 2}}}{{{q^{m - 2}}}}{\partial _t^{m - 1}{w}(1,t)}^2\notag\\
&\quad-\left(\bar R_1\xi_1- \frac{{r_{2m - 1}}{a_{2m - 1}}}{{{q^{m - 1}}}}\right)\eta_m(t)^2-\bar R_0\lambda V(t).
\end{align}
Choosing large enough $\bar R_0,\bar R_1$, we obtain
\begin{align}
\dot V_{u1}(t)\le -\lambda_{u1}V_{u1}(t)\label{eq:dvu1}
\end{align}
for some positive $\lambda_{u1}$, where $-y_1^2,\cdots,-y_m^2$ in $-\frac{\bar R_0}{2}\lambda V(t)$ extracted from $-{\bar R_0}\lambda V(t)$ are used to ``cancel'' ${\partial _t^{m - 1}{w}(1,t)}^2$ by recalling \eqref{eq:norm2}-\eqref{eq:tVt} and \eqref{eq:nct1}-\eqref{eq:nctm}.

Therefore, we obtain the exponential stability estimate in the sense of $(\|\partial_x^{2m-2}u(\cdot,t)\|^2+\|\partial_x^{2m-1}u(\cdot,t)\|^2)^{\frac{1}{2}}$ by recalling \eqref{eq:dvu1}, \eqref{eq:Vmn1}, \eqref{eq:V2m2} and the invertibility of the backstepping transformation.

\textbf{Next, we prove the exponential stability estimate in the sense of $(\|\partial_x^{2m}u(\cdot,t)\|^2+\|\partial_x^{2m+1}u(\cdot,t)\|^2)^{\frac{1}{2}}$.}

Taking $2m$ and $2m+1$ times derivative of \eqref{eq:target2} with respect to $x$ respectively, we have
\begin{align}
\partial_x^{2m}w_t(x,t)&=q\partial_x^{2m+2}w(x,t),\label{eq:2m}\\
\partial_x^{2m+1}w_t(x,t)&=q\partial_x^{2m+3}w(x,t).\label{eq:2m1}
\end{align}

Consider now a Lyapunov function
\begin{align*}
{V_{{\rm{2m}}}}(t) = \frac{{{a_{2m}}}}{2}\| {\partial _x^{2m}{w}} (\cdot,t)\|^2+ \frac{{{a_{2m + 1}}}}{2}\| {\partial _x^{2m + 1}{w}} (\cdot,t)\|^2,
\end{align*}
where the positive parameters $a_{2m},a_{2m+1}$ are to be chosen later.
Taking the derivative of $V_{2m}(t)$ along the \eqref{eq:2m}-\eqref{eq:2m1} as
\begin{align*}
&\dot V_{2m}(t)\le- \bigg(({a_{2m}} - {a_{2m + 1}})q-\frac{{a_{2m}}}{{4{q^{m - 1}}{r_{2m}}}} \notag\\
 &\quad- \frac{{a_{2m + 1}}}{{4{q^m}{r_{2m+1}}}}\bigg){\left\| {\partial _x^{2m + 1}w} \right\|^2}- \bigg(\frac{1}{2}{a_{2m + 1}}q \notag\\
 &\quad- \frac{{a_{2m}}}{{4{q^{m - 1}}{r_{2m}}}}- \frac{{a_{2m + 1}}}{{4{q^m}{r_{2m+1}}}}\bigg){\left\| {\partial _x^{2m + 2}w} \right\|^2}\notag\\
 &\quad- \bigg(\frac{{{a_{2m + 1}}q}}{2}-\frac{{a_{2m}}}{{4{q^{m - 1}}{r_{2m}}}}- \frac{{a_{2m + 1}}}{{4{q^m}{r_{2m+1}}}}\bigg){\partial _x^{2m + 1}{w}(0,t)}^2\notag\\
 &\quad+\frac{ {r_{2m}}{a_{2m}}}{{{q^{m - 1}}}}{\partial _t^mw(1,t)}^2+\frac{{r_{2m+1}}{a_{2m + 1}}}{{{q^m}}}{\partial _t^{m + 1}w(1,t)}^2.
\end{align*}
We choose $a_{2m}>a_{2m+1}$ and sufficiently large ${r_{2m}}$, ${r_{2m+1}}$ to make
\begin{align*}
&\dot V_{2m}(t)\le- \frac{1}{5}\bigg(({a_{2m}} - {a_{2m + 1}})q-\frac{{a_{2m}}}{{4{q^{m - 1}}{r_{2m}}}} \notag\\
 &\quad- \frac{{a_{2m + 1}}}{{4{q^m}{r_{2m+1}}}}\bigg)\left({\left\| {\partial _x^{2m}w} \right\|^2}+{\left\| {\partial _x^{2m + 1}w} \right\|^2}\right)\notag\\
 &\quad+\frac{ {r_{2m}}{a_{2m}}}{{{q^{m - 1}}}}{\partial _t^mw(1,t)}^2+\frac{{r_{2m+1}}{a_{2m + 1}}}{{{q^m}}}{\partial _t^{m + 1}w(1,t)}^2,
\end{align*}
where Poincar$\acute e$ inequality is used.

We thus have
\begin{align}
\dot V_{2m}(t)&\le- \bar\lambda_{2m} V_{2m}(t)+\frac{ {r_{2m}}{a_{2m}}}{{{q^{m - 1}}}}{\partial _t^mw(1,t)}^2\notag\\
 &\quad+\frac{{r_{2m+1}}{a_{2m + 1}}}{{{q^m}}}{\partial _t^{m + 1}w(1,t)}^2,\label{eq:dVmfinal}
\end{align}
for some positive $\bar\lambda_{2m}$.

Taking the time derivative of \eqref{eq:wtm} and using Cauchy-Schwarz inequality, we have
\begin{align}
\left|\partial _t^{m+1}w(1,t)\right|^2\le& \xi_0\|w_{xx}(\cdot,t)\|^2+\xi_0|X(t)|^2\notag\\
&+\xi_0w_x(0,t)^2+\xi_0\sum_{i=1}^{m}\left|\partial _t^iw(1,t)\right|^2\label{eq:dwm1}
\end{align}
for some positive $\xi_0$, where \eqref{eq:target1} is used. Recalling \eqref{eq:norm2}, \eqref{eq:omega1}-\eqref{eq:V2bound}, \eqref{eq:dv2final}, we have the exponential stability estimate in the sense of $\sum_{i=0}^{2}\|\partial _x^iw(\cdot,t)\|$. Using Sobolev inequality, we obtain the exponential stability estimate in
term of the norm $\|w(\cdot,t)\|_{C_1}$. Together with \eqref{eq:nct1}-\eqref{eq:nctm}, \eqref{eq:norm2}, \eqref{eq:omega1}-\eqref{eq:V2bound}, \eqref{eq:dv2final} and \eqref{eq:etam}, we have the exponential convergence of $\partial _t^{m+1}w(1,t)$ via \eqref{eq:dwm1}, i.e.,
\begin{align}
\left|\partial _t^{m+1}w(1,t)\right|\le \bar\Upsilon_{m+1}e^{-\xi_2t}=\eta_{m+1}(t),\label{eq:etam1}
\end{align}
for some positive $\bar\Upsilon_{m+1}$ and $\xi_2$.

Define a Lyapunov function
\begin{align}
{V_{u}}(t) &= V_{2m}(t)+\frac{R_1}{2}\eta_{m}(t)^2+\frac{R_2}{2}\eta_{m+1}(t)^2,\label{eq:Vmn}
\end{align}
where $R_1,R_2$ are positive constants to be determined later.

Defining
\begin{align}
\Omega_{m}(t)&=\|\partial_x^{2m} w(\cdot,t)\|^2+\|\partial_x^{2m+1} w(\cdot,t)\|^2\notag\\
&\quad+\eta_{m}(t)^2+\eta_{m+1}(t)^2,\label{eq:Omegam}
\end{align}
we have
\begin{align}
\theta_{m1}\Omega_m(t)\le V_{u}(t)\le \theta_{m2}\Omega_m(t),\label{eq:Omegam1}
\end{align}
where ${\theta _{m1}} = \min \left\{ \frac{a_{2m}}{2},\frac{a_{2m+1}}{2},\frac{R_1}{2},\frac{R_2}{2}\right\}$, ${\theta _{m2}} = \max \left\{ \frac{a_{2m}}{2},\frac{a_{2m+1}}{2},\frac{R_1}{2},\frac{R_2}{2}\right\}$.

Taking the derivative of \eqref{eq:Vmn}, recalling \eqref{eq:dVmfinal}, \eqref{eq:etam}, \eqref{eq:etam1}, we obtain
\begin{align}
\dot V_{u}(t)&\le- \bar\lambda_{2m} V_{2m}(t)-\bigg(R_1\xi_1-\frac{ {r_{2m}}{a_{2m}}}{{{q^{m - 1}}}}\bigg)\eta_m(t)^2\notag\\
 &\quad-\bigg(R_2\xi_2-\frac{{r_{2m+1}}{a_{2m + 1}}}{{{q^m}}}\bigg)\eta_{m+1}(t)^2.\label{eq:dVmnfinal}
\end{align}
Choosing large sufficiently $R_1,R_2$, we arrive at
\begin{align}
 \dot V_{u}(t)\le- \lambda_{u} V_{u}(t)
\end{align}
for some positive $\lambda_{u}$.

Now we obtain the exponential stability estimate in the sense of the norm $\Omega_{m}(t)^{\frac{1}{2}}$ recalling \eqref{eq:Omegam}-\eqref{eq:Omegam1}. Using the invertibility of \eqref{eq:t1}, we obtain the exponential stability estimate in the sense of $(\|\partial_x^{2m}u(\cdot,t)\|^2+\|\partial_x^{2m+1}u(\cdot,t)\|^2)^{\frac{1}{2}}$.

Therefore, using the result in the subsection $1)$ and the above proof in the subsection $2)$, we obtain $ii)$ of Lemma \ref{lem:2}.

The proof of Lemma \ref{lem:2} is completed.

\textbf{\emph{Proof of Theorem \ref{lem:erm}:}} Define a Lyapunov function
\begin{align}
V_w(t)=\sum\limits_{i=0}^{2m}\|\partial_x^{i}\tilde w(\cdot,t)\|^2.\label{eq:Vw}
\end{align}

Taking $i$ times derivative of \eqref{eq:w2} with
respect to $x$,
\begin{align}
\partial_x^{i}\tilde w_t(x,t)=q\partial_x^{i}\tilde w_{xx}(x,t),\label{eq:ix}
\end{align}
where $i=1,\cdots,2m$.
Taking the derivative of \eqref{eq:Vw} along \eqref{eq:w2}-\eqref{eq:w4} and \eqref{eq:ix}, applying Poincar$\acute e$ inequality, yields
\begin{align}
\dot V_w(t)&=-q\sum\limits_{i=1}^{2m+1}\|\partial_x^{i}\tilde w(\cdot,t)\|^2\notag\\
&\le-\frac{1}{5}q\sum\limits_{i=0}^{2m}\|\partial_x^{i}\tilde w(\cdot,t)\|^2\le -\lambda_w V_w(t),
\end{align}
for some positive $\lambda_w$.

We thus obtain
\begin{align}
\sum_{i=0}^{2m}\|\partial_x^{i}\tilde w(\cdot,t)\|^2\le&\Upsilon_{\tilde w}\bigg(\sum_{i=0}^{2m}\|\partial_x^{i}\tilde u(\cdot,0)\|^2\notag\\
&+\left|\tilde Z(0)\right|^2+\left|\tilde X(0)\right|^2\bigg)e^{-\lambda_{\tilde w}t},\label{eq:tw}
\end{align}
with $\Upsilon_{\tilde w},\lambda_{\tilde w}$ are some positive constants, where \eqref{eq:tra1} is used to replace $\sum_{i=0}^{2m}\|\partial_x^{i}\tilde w(\cdot,0)\|^2$.

Especially from the exponential stability result of $\sum_{i=0}^{2}\|\partial_x^{i}\tilde w(\cdot,t)\|$ and using Sobolev inequality, we can obtain the exponential stability estimate in
term of the norm $\|\tilde w(\cdot,t)\|_{C_1}$, which gives the exponential
convergence of $\tilde w_x(0,t)$. Considering \eqref{eq:w5}, because $\bar A$ \eqref{eq:barA} is Hurwitz and $\tilde w_x(0,t)$ is exponentially convergent, $|\tilde Z(t)|,|\tilde X(t)|$ are exponentially convergent in the sense of
\begin{align}
&\quad|\tilde Z(t)|^2+|\tilde X(t)|^2\notag\\
&\le \Upsilon_{ZX}\left(|\tilde Z(0)|^2+|\tilde X(0)|^2+\tilde u_x(0,0)^2\right)e^{-\lambda_{ZX}t}\label{eq:ZX}
\end{align}
with $\Upsilon_{ZX},\lambda_{ZX}$ being some positive constants, where \eqref{eq:tra1} is used to replace $\tilde w_x(0,0)^2$.

Recalling \eqref{eq:tra1} and applying Cauchy-Schwarz inequality, we have
\begin{align}
\sum_{i=0}^{2m}\|\partial_x^{i}\tilde u(\cdot,t)\|^2\le&3\sum_{i=0}^{2m}\|\partial_x^{i}\tilde w(\cdot,t)\|^2+3\sum_{i=0}^{2m}\|\partial_x^{i}\vartheta(\cdot)\|^2\left|\tilde Z(t)\right|^2\notag\\
&+3\sum_{i=0}^{2m}\|\partial_x^{i}\theta(\cdot)\|^2\left|\tilde X(t)\right|^2.
\end{align}
Because $\sum_{i=0}^{2m}\|\partial_x^{i}\vartheta(\cdot)\|^2$ and $\sum_{i=0}^{2m}\|\partial_x^{i}\vartheta(\cdot)\|^2$ are bounded gains according to \eqref{eq:ss1}-\eqref{eq:ss2}, applying \eqref{eq:tw}-\eqref{eq:ZX}, we obtain \eqref{eq:lemma4norm}-\eqref{eq:ernorm}.
\ifCLASSOPTIONcaptionsoff
  \newpage
\fi


%

%




\end{document}